\documentclass[10pt]{amsart}
\usepackage[all]{xy}
\usepackage{amsmath,amscd,amsfonts,amsthm,amssymb,latexsym,mathrsfs}
\usepackage{graphicx,epsfig}
\usepackage{url,hyperref}

\newcommand{\mf}{\mathfrak}
\newcommand{\mc}{\mathcal}

\newcommand{\bb}{\mathbb}
\newcommand{\T}{\mathbf{T}}
\newcommand{\M}{\mathbf{M}}
\newcommand{\Mg}{M_{\text{gen}}}
\newcommand{\Mm}{M_{\mf{m}}}

\newcommand{\noequals}{&\mathrel{\phantom{=}}} 

\DeclareMathOperator{\coker}{{\rm coker}}

\newcommand{\id}{{\rm id}}

\DeclareMathOperator{\Hom}{{\rm Hom}}
\DeclareMathOperator{\Ext}{{\rm Ext}}
\DeclareMathOperator{\Tor}{{\rm Tor}}

\DeclareMathOperator{\rk}{{\rm rk}}
\DeclareMathOperator{\rank}{{\rm rank}}
\DeclareMathOperator{\cork}{{\rm cork}}
\DeclareMathOperator{\Frac}{{\rm Frac}}
\DeclareMathOperator{\Pic}{{\rm Pic}}
\DeclareMathOperator{\cl}{{\rm det}} 
\newcommand{\tors}[1]{#1_{\rm tors}} 
\newcommand{\projpart}[1]{#1_{\rm proj}}
\DeclareMathOperator{\val}{{\rm val}}

\DeclareMathOperator{\conv}{{\rm conv}}

\newcommand{\fgMod}[1]{\text{\rm $#1$-Mod}}
\newcommand{\Mat}[1]{\text{\rm $#1$-Mat}}

\newcommand{\K}{\bb K}
\newcommand{\Z}{\bb Z}
\newcommand{\Q}{\bb Q}

\newcommand{\C}{\bb C}

\newcommand{\<}{\langle}
\renewcommand{\>}{\rangle}

\let\oldmarginpar\marginpar
\renewcommand\marginpar[1]{\-\oldmarginpar[\footnotesize #1]{\raggedright\footnotesize #1}}

\DeclareRobustCommand{\qedify}[1]{%
  \ifmmode \quad\hbox{#1}
  \else
    \leavevmode\unskip\penalty9999 \hbox{}\nobreak\hfill
    \quad\hbox{#1}%
  \fi
}
\newtheorem{theorem}{Theorem}[section]

\newtheorem{lemma}[theorem]{Lemma}
\newtheorem{proposition}[theorem]{Proposition}
\newtheorem*{propchar2}{Proposition \ref*{p:DVR M2}a}
\newtheorem{corollary}[theorem]{Corollary}

\theoremstyle{definition}
\newtheorem{definition}[theorem]{Definition}
\newtheorem{Example}[theorem]{Example}
\newenvironment{example}[1][]{\begin{Example}[#1]\pushQED{\qedify{$\diamondsuit$}}}{\popQED\end{Example}}
\newtheorem*{Exchar2}{Example \ref*{p:DVR M2}b}
\newenvironment{exchar2}{\begin{Exchar2}\pushQED{\qedify{$\diamondsuit$}}}{\popQED\end{Exchar2}}
\newtheorem{fact}[theorem]{Fact}
\theoremstyle{remark}
\newtheorem{Remark}[theorem]{Remark}
\newenvironment{remark}[1][]{\begin{Remark}[#1]\pushQED{\qedify{$\diamondsuit$}}}{\popQED\end{Remark}}
\newtheorem{Question}[theorem]{Question}
\newenvironment{question}[1][]{\begin{Question}[#1]\pushQED{\qedify{$\diamondsuit$}}}{\popQED\end{Question}}

\numberwithin{equation}{section}

\begin{document}

\title{Matroids over a ring}

\author[A. Fink]{Alex Fink$^\ast$}
\thanks{$^\ast$Queen Mary University of London}
\address{Alex Fink\\School of Mathematical Sciences, Queen Mary University of London, 327 Mile End Road, London E1 4NS, United Kingdom
}
\email{a.fink@qmul.ac.uk}

\author[L.~Moci]{Luca Moci$^\dag$}
\thanks{$^{\dag}$ IMJ -- Universit\'{e} de Paris 7. Supported by a Marie Curie Fellowship of Istituto Nazionale di Alta Matematica}
\address{Luca Moci\\IMJ -- Universit\'{e} de Paris 7, 5 rue Thomas Mann, 75205 Paris Cedex 13, France}
\email{moci@math.jussieu.fr}

\begin{abstract}
We introduce the notion of a matroid $M$ over a commutative ring $R$, assigning to every subset of the ground set an $R$-module according to some axioms. When $R$ is a field, we recover matroids.  When $R=\mathbb{Z}$, and when $R$ is a DVR, we get (structures which contain all the data of)
quasi-arithmetic matroids, and valuated matroids i.e.\ tropical linear spaces, respectively.

More generally, whenever $R$ is a Dedekind domain, we extend all the usual properties and operations holding for matroids (e.g., duality), and we explicitly describe the structure of the matroids over $R$. Furthermore, we compute the Tutte-Grothendieck ring of matroids over $R$. We also show that the Tutte quasi-polynomial of a matroid over $\mathbb{Z}$ can be obtained as an evaluation of the class of the matroid in the Tutte-Grothendieck ring.

\end{abstract}
\maketitle

\section{Introduction}\label{sec:introduction}
The notion of a \emph{matroid} axiomatizes the linear algebra of a list of 
vectors.  Matroid theory has proved to be a versatile language to deal with many 
problems on the interfaces of combinatorics and algebra. 
In the years since 1935, when Whitney first introduced matroids,
a number of enriched variants thereof have arisen,
among them oriented matroids~\cite{BlV}, valuated matroids~\cite{DW}, 
complex matroids~\cite{AD}, and (quasi-)arithmetic matroids~\cite{MociT, D'Adderio-Moci}.
Each of these structures retains some information about
a vector configuration, or an equivalent object,
which is richer than the purely linear algebraic information
that matroids retain.

As a running motivating example, let us focus on quasi-arithmetic matroids.
A quasi-arithmetic matroid endows a matroid with a multiplicity function, 
whose values are the cardinalities of certain finite abelian groups, 
namely, the torsion parts of the quotients of an ambient lattice 
$\mathbb{Z}^n$ by the sublattices spanned by subsets of vectors.
From a list of vectors with integer coordinates
one may produce objects like a toric 
arrangement, a partition function, and a zonotope (see \cite{DP-book}).
In order to have a combinatorial structure from which
these objects may be read off, one needs to keep track of
arithmetic properties of the vectors, and this is what
quasi-arithmetic matroids provide.
(For the difference between quasi-arithmetic and 
arithmetic matroids, see Remark~\ref{rem:quasi}.)

It is natural to ask to what extent these generalizations of matroids
can be unified under a common framework.  Such a unification
was sought by Dress in his program of 
matroids with coefficients, represented for example in 
his work with Wenzel \cite{DW} wherein
valuated matroids are matroids with coefficients in a ``fuzzy ring''.

In the present paper we suggest a different approach to such unification,
by defining the notion of a \emph{matroid $M$ over a 
commutative ring $R$}.  Such an~$M$ assigns, to every subset $A$ of a ground 
set, a finitely generated $R$-module $M(A)$ according to some axioms (Definition 
\ref{def:matroid}).  We find this definition to have multiple agreeable features.
For one, by building on the well-studied setting of 
modules over commutative rings, we get a theory
where the considerable power and development 
of commutative algebra can be easily brought to bear.
For another, unlike arithmetic and valuated matroids,
a matroid over~$R$ is not defined as a matroid decorated with extra data;
there is only one axiom, and we suggest that it is comparably simple
to the matroid axioms themselves. Indeed, a \emph{realizable} matroid over~$R$
is precisely a vector configuration in a finitely generated $R$-module, 
and the axioms of a matroid over~$R$
say only that minors of at most two elements are such realizable matroids ---
that is, matroids are \emph{locally} realizable matroids.

When $R$ is a field, a matroid $M$ over~$R$ is nothing but a matroid:
the datum $M(A)$ is a vector space, which contains only the
information of its dimension, and this directly encodes the rank function of~$M$. 
When $R=\mathbb{Z}$, every module $M(A)$ is an 
abelian group, and by extracting its torsion subgroup we get a quasi-arithmetic matroid. 
When $R$ is a discrete valuation ring (DVR), we may similarly extract 
a valuated matroid.  
More generally, whenever $R$ is a Dedekind domain, we can
extend the usual properties and operations holding for matroids,
such as duality. 

The idea of matroids over rings was suggested by certain 
features of the theory of quasi-arithmetic matroids.  
Some significant information about an integer vector configuration
is lost in passing to the multiplicity function,
as there exist many finite abelian groups with the same cardinality.
Recording the whole structure of these groups is more desirable in several 
situations, for example, in developing a combinatorial intersection 
theory for the arrangements of subtori arising as characteristic 
varieties.
The properties of the multiplicity function of a quasi-arithmetic matroid 
turn out to be just shadows of group-theoretic properties.

One of the most-loved invariants of matroids is their Tutte polynomial $\T_M(x,y)$.
It thus comes as no surprise that the Tutte polynomial has been considered
for generalizations of matroids as well.
A quasi-arithmetic matroid $\hat{M}$ has an associated
arithmetic Tutte polynomial $\M_{\hat{M}}(x,y)$, which has proved to be an useful tool in 
studying toric arrangements, partition functions, zonotopes, and graphs 
(\cite{MociT, DM-g, Branden-Moci}).  More strongly, 
the authors of \cite{Branden-Moci} define a \emph{Tutte quasi-polynomial}
of an integer vector configuration, 
interpolating between $\T_M(x,y)$ and $\M_{\hat{M}}(x,y)$, which is no longer
an invariant of the quasi-arithmetic matroid (as it depends on the groups, not
just their cardinalities).

Among its properties, the Tutte polynomial of a classical matroid
is the universal deletion-contraction invariant. 
In more algebraic language, following \cite{Brylawski}, 
the class of a matroid in the Tutte-Grothendieck ring 
for deletion-contraction relations is exactly its Tutte polynomial. 
While the arithmetic Tutte polynomial and Tutte quasi-polynomial
are deletion-contraction invariants, neither is universal for this property.
Our generalization of the Tutte polynomial for matroids over a Dedekind ring~$R$
is also the class in the Tutte-Grothendieck ring, so it retains 
the universality of the usual Tutte polynomial, and we obtain the two generalizations
of Tutte just mentioned as evaluations of~it.

This paper is organized as follows. In Section~\ref{sec:definition} we give the basic definitions 
for matroids over a commutative ring, including realizability,
and we explain how they generalize the classical ones. 

In Section~\ref{sec:duality}, we 
establish the existence (Definition~\ref{def:dual}, Proposition~\ref{p:dual}) 
and the properties of the dual of a matroid over a Dedekind domain.  
The case of Dedekind domains
is the one we focus primarily on thereafter, and 
we review the properties of such rings in Section~\ref{sec:Dedekind}.

In Section~\ref{sec:DVR} we develop the local theory, by  proving a structure theorem for 
matroids over a DVR (Propositions \ref{p:DVR M1} and~\ref{p:DVR M2}). 
We show connections with the Hall algebra and with tropical 
geometry.  A matroid over a DVR defines a point on each Dressian, one of the
tropical analogues of the Grassmannian; 
this is equivalent by definition to being a valuated matroid.  
In fact, such a matroid defines a point on the corresponding analogue of
the full flag variety (Corollary~\ref{conj:DVR exchange}).

The global theory is developed is Section~\ref{sec:global}. 
We describe the structure of a matroid over a Dedekind ring~$R$ in terms of that of all its localizations, 
whose structure was completely described in the previous section,
plus some global information coming from the Picard group of~$R$ 
(Propositions \ref{p:Dedekind M1} and~\ref{p:Dedekind M2}).
This also explains the connection between matroids 
over $\mathbb{Z}$ and quasi-arithmetic matroids (Corollary \ref{c:arithmetic}).

In Section~\ref{sec:Grothendieck} we compute the Tutte-Grothendieck ring 
(Theorem~\ref{th:Grothendieck}, Corollary~\ref{cor:what the ring is exactly}).
In particular, given a matroid over 
$\mathbb{Z}$, we present its Tutte quasi-polynomial as an
evaluation of its class in~$K(\Mat{\mathbb{Z}})$.

\subsection*{Acknowledgments}
The authors thank Ezra Miller for helpful conversations,
and the anonymous referee for useful and thought-provoking suggestions.

\section{Matroids over a ring}\label{sec:definition}

By $\fgMod{R}$ we mean the category of finitely generated $R$-modules
over a commutative ring $R$.  
We will feel free to
write ``f.g.''\ for ``finitely generated'' throughout.

\begin{definition}\label{def:matroid}
Let $R$ be a commutative ring. 
A \emph{matroid over $R$} on the ground set $E$
is a function $M$ assigning to each subset $A\subseteq E$
a finitely-generated $R$-module $M(A)$ such that
\begin{itemize}
\item[(M)] for every subset $A\subseteq E$ and elements
$b,c\in E$, there exist elements
$x=x(b,c)$ and $y=y(b,c)$ of $M(A)$ satisfying
\begin{align*}
M(A\cup\{b\}) &\cong M(A)/(x) \\
M(A\cup\{c\}) &\cong M(A)/(y) \\
M(A\cup\{b,c\}) &\cong M(A)/(x,y).
\end{align*}
\end{itemize}
\end{definition}

Clearly, the choice of the modules $M(A)$ is only relevant up to isomorphism.
We regard matroids $M$ and $M'$ over~$R$ to be equal if they are on the same 
ground set $E$ and $M(A)\cong M'(A)$ for all $A\subseteq E$.

For notational concision, we will hereafter let $M(Ab)$ abbreviate $M(A\cup\{b\})$,
$M(Abc)$ stand for $M(A\cup\{b, c\})$, and so forth.

The case of axiom (M) where $b=c$ is the following statement,
which we separate out here as it will provide a useful waypoint in
many of the proofs to come.
\begin{itemize}
\item[(M1)] for every subset $A\subseteq E$ and element
$b\in E$, there exists an element
$x=x(b)$ of $M(A)$ such that $M(A\cup\{b\}) \cong M(A)/(x)$.
\end{itemize}

A more abstract, but equivalent definition will be given in Section \ref{tg}.

The fundamental way of producing matroids over $R$ is from vector configurations in an $R$-module.
Given a f.g.\ $R$-module $N$ and a list $X={x_1,\ldots,x_n}$ of elements of $N$, the matroid $M_X$ of~$X$
associates to the sublist $A$ of $X$ the quotient 
\begin{equation}\label{eq:matroid from vectors}
  M_X(A) = N\Big/\left(\sum_{x\in A}Rx\right).
\end{equation}
For each $x\in X$
there is a quotient map from $M_X(A)$ to $M_X(A\cup\{x\})$, which quotients out by
the image of $Rx$ in $M_X(A)$.  
This single system of maps satisfies axiom~(M):
indeed, the element $x$ depends only on $b$, and $y$ only on~$c$.

\begin{definition}\label{def:realizable}
A matroid over $R$ is \emph{realizable} (or \emph{representable}) if it 
has the form $M_X$ for some list $X$ of elements
of a f.g.\ $R$-module.  We call $X$ a \emph{realization} (\emph{representation})
of~$M$.
\end{definition}

Not all matroids over rings are realizable: indeed, nonrealizable
matroids in the usual sense will provide examples.
Axiom (M) requires only a sort of ``local'' realizability.
Allowing the elements $x$ and $y$ in the axiom to depend on both
$b$ and $c$ is what prevents this local realizability from immediately
extending to global realizability.

\begin{example}
The following four abelian groups do not form a matroid over~$\Z$:
\[M(\emptyset)=\Z/8, \quad M(\{1\}) = \Z/2, \quad M(\{2\}) = \Z/2, \quad M(\{1,2\}) = 1.\]
In axiom (M), when $(A,b,c)=(\emptyset, 1, 2)$, the elements $x$ and $y$
must both be chosen to be in the subgroup of $\Z/8$ generated by 2 for the 
isomorphisms for $M(\{1\})$ and $M(\{2\})$ to hold, but then the isomorphism
for $M(\{1,2\})$ fails.

The following four abelian groups do form a matroid over~$\Z$:
\[M(\emptyset)=\Z/4\oplus \Z/2, \quad M(\{1\}) = \Z/2, \quad M(\{2\}) = \Z/2, \quad M(\{1,2\}) = 1.\]
Indeed, it is a realizable matroid, with a realization being
given by $N=\Z/4\oplus \Z/2$, $x_1 = (1,0)$, $x_2 = (1,1)$.
\end{example}

Our having chosen to call these objects ``matroids over $R$'' is appropriate,
as they are a generalization of matroids in the classical sense,
as we show in Proposition~\ref{p:classical}.
There is one hitch in the equivalence, corresponding to the ability to
choose a vector configuration that does not span its ambient space.
Accordingly, let us say that a matroid $M$ over $R$ is \emph{essential}
if no nontrivial projective module is a direct summand of $M(E)$
(the term is adopted from the theory of hyperplane arrangements).
Lemma~\ref{lem:essential} shows that very little is lost in restricting to essential matroids.

Before getting there we must generalize some standard operations on matroids.
In several cases this is straightforward, but duality is conspicuously not among these:
for matroid duality to work well, we must assume that $R$ is a Dedekind domain,
and so we treat it in Section~\ref{sec:duality}.

Let $M$ and $M'$ be matroids over~$R$ on respective ground sets $E$ and~$E'$.
We define their \emph{direct sum} $M\oplus M'$ on the ground set $E\amalg E'$ by
\[(M\oplus M')(A\amalg A') = M(A)\oplus M'(A').\]
If $i$ is an element of $E$, we define two
matroids over~$R$ on the ground set $E\setminus\{i\}$:
the \emph{deletion} of $i$ in $M$, denoted $M\setminus i$, by
\[(M\setminus i)(A) = M(A)\]
and the \emph{contraction} of $i$ in $M$, denoted $M\setminus i$, by
\[(M/i)(A) = M(A\cup\{i\}).\]
It is easy to check that these satisfy the axioms (M);
this is entirely inherited except for~$M\oplus M'$
when one of $b$ and~$c$ in $E$ and the other is in $E'$, but these cases are clear.
Since these constructions can be made without reliance on the axiom (M),
we will sometimes use them in the sequel when speaking of a map $\mc B(E)\to\fgMod{R}$
which has not yet been shown to be a matroid over~$R$.

The next fact is immediate from these definitions.
\begin{fact}
The class of realizable matroids is closed under minors and direct sums:
\begin{enumerate}\renewcommand{\labelenumi}{(\alph{enumi})}
\item If $M$ is realized by the vector configuration $X$
within a module $N$, then $M/A\setminus B$ is realized
by the images of the vectors $x_i$ where $i\in E\setminus(A\cup B)$
in the quotient module $N/(x_i : i\in A)$.
\item If $M_i$ is realized by the vector configuration $X_i$
within a module~$N_i$, for $i=1,2$, then
$M_1\oplus M_2$ is realized by the configuration 
$(X_1,0)\cup(0,X_2)$
within $N_1\oplus N_2$.
\end{enumerate}
\end{fact}

If $N$ is an $R$-module, let the \emph{empty matroid} for $N$
be the matroid over~$R$ on the ground set $\emptyset$ which maps $\emptyset$ to~$N$.
By a \emph{projective empty matroid} we mean an empty matroid for a projective module.

\begin{lemma}\label{lem:essential}
Every matroid $M$ over~$R$ is the direct sum of
an essential matroid over~$R$ and a projective empty matroid.
\end{lemma}

Note that this decomposition is unique if $R$ is a field,
or if $R$ is Dedekind (Proposition~\ref{p:Dedekind structure}).

\begin{proof}
Suppose $M$ is not essential, so that some projective module $P$ is a direct summand of $M(E)$.
Then in fact $P$ is a direct summand of every
module $M(A)$, since this property lifts back along
the surjections $M(A)\twoheadrightarrow M(Ab)$.  Therefore $M$ is a direct sum
of another matroid $M'$ over $R$ and the empty matroid for $P$.
Since $M$ is finitely generated, iterating this process with
$M'$ in place of~$M$ eventually reaches an essential matroid.
\end{proof}

Recall that the \emph{corank}
$\cork(A)$ of a set $A$ in a classical matroid is
equal to $\rk(E) - \rk(A)$, where $\rk(E)$ is the rank of the matroid.

\begin{proposition}\label{p:classical}
Let $\K$ be a field.
Essential matroids $M$ over $\K$ are equivalent to (classical) matroids.
If $M$ is an essential matroid over $\K$, then $\dim M(A)$ is the corank of $A$ in the
corresponding classical matroid.

A matroid over $\K$ is realizable if and only if, as a classical matroid,
it is realizable over $\K$.
\end{proposition}

\begin{proof}
The finitely generated modules over $\K$ are the
finite-dimensional $\K$-vector spaces, which are completely classified up to isomorphism
by dimension.  So we may replace $M(A)$ by its $\K$-dimension without losing information.

We now check that the conditions on the dimensions of the $M(A)$
given by axiom (M) and the essential condition are equivalent to
the following set of rank axioms for matroids, recast in terms of a
corank function $\cork:2^E\to\bb N$:
\begin{enumerate}\renewcommand{\labelenumi}{(C\arabic{enumi})}\setcounter{enumi}{-1}
\item $\cork(E) = 0$.
\item For $A\subseteq E$ and $b\in E\setminus A$, $\cork(A)-\cork(Ab)$ equals 0 or 1.
\item For $A\subseteq E$ and $b\neq c\in E\setminus A$,
\[\cork(A)+\cork(Abc)\geq\cork(Ab)+\cork(Ac).\]
\end{enumerate}
Axiom (C0) is the ``essential'' condition.
For $x$ in a $\K$-vector space $V$, the difference $\dim V-\dim(V/\<x\>)$
equals zero if $x$ is zero and one otherwise, so that (C1) is equivalent
to~(M1).

Finally, in axiom (M), let the singly-generated subspaces
$K=(x)$ and~$L=(y)$ be the respective kernels of
$M(A)\to M(Ab)$ and $M(A)\to M(Ac)$.  Then $M(Abc) = M(A) / (K\cup L)$.  
By arranging $K$ and $L$ suitably, their sum $K + L$
can be chosen to have any dimension from $\max(\dim K,\dim L)$
to $\dim K+\dim L$ inclusive (except those that exceed $\dim M(A)$), but no others.
That is, the only conditions on $\dim M(Abc)$ in terms of the other dimensions are
the monotonicity conditions $\dim M(Abc)\leq\min(\dim M(Ab), \dim M(Ac))$,
and the submodularity condition
\[\dim M(A)+\dim M(Abc)\geq \dim M(Ab)+\dim M(Ac),\]
which is (C2).  
Since (C1) and/or (M1) 
implies the monotonicity conditions, we have the desired equivalence.

The realizability claim is already proved by
our prior observation that a realized matroid over~$\K$
embodies a $\K$-vector configuration.
\end{proof}

Let $R\to S$ be a map of rings. Then every matroid over $S$ is naturally also a matroid over $R$. Furthermore, given such a map $R\to S$, the tensor product $\mbox{---}\otimes_R S$
is a functor $\fgMod{R}\to\fgMod{S}$.  One can use this to perform
\emph{base change} of matroids over~$R$.
If $M$ is a matroid over~$R$, define $M\otimes_R S$ be the composition of
$M$ with $\mbox{---}\otimes_R S$, so that
\[(M\otimes_R S)(A) = M(A)\otimes_R S\] for all~$A$.
As with other uses of the tensor product, we will omit the subscript $R$ in the notation where
this causes no unclarity.

\begin{proposition}\label{p:base change}
If $M$ is a matroid over $R$, then $M\otimes_R S$ is a matroid over $S$.
\end{proposition}

\begin{proof}
Let $0\to K\to N\to N'\to 0$ be a short exact sequence of $R$-modules,
with $K$ cyclic.
Tensor product being right exact, we get an exact sequence
$K\otimes S\to N\otimes S\to N'\otimes S\to 0$, so the kernel of
$N\otimes S\to N'\otimes S$ is a quotient of the cyclic $S$-module $K\otimes S$,
and is therefore cyclic.  Therefore the maps from $M(A)\otimes S$
to $M(Ab)\otimes S$ and to $M(Ac)\otimes S$ have cyclic kernels,
and this establishes condition (M1) for $M\otimes S$.

Since tensor product is a left adjoint functor (to $\Hom$), it preserves pushouts,
including the pushout from axiom (M) for $M$:
\[\xymatrix{
\ar@{}[dr]|{\mbox{\LARGE$\lrcorner$}}
M(A    )\ar[r]\ar[d] & M(Ab)\ar[d] \\
M(Ac)\ar[r] & M(Abc)
}\]
This proves axiom (M) for $M\otimes S$.
\end{proof}

Two special cases of this construction will be of fundamental importance for our theory.
\begin{enumerate}
\item For every prime ideal $\mf{m}$ of $R$, let $R_{\mf m}$ be the localization of $R$ at $\mf m$. We call 
$$M_{\mf m}\doteq M\otimes_R R_{\mf m}$$
the \emph{localization} of $M$ at $\mf m$.
\item If $R$ is a domain, let $\Frac(R)$ be the fraction field of $R$. Then we call
$$M_{\text{gen}}\doteq M\otimes_R \Frac(R)$$
the \emph{generic matroid} of $M$.
\end{enumerate}

A \emph{generic loop} or \emph{generic coloop} of a matroid over $R$ 
is a loop, respectively a coloop, of its generic matroid.  
Thus $a$ is a generic coloop of $M$ if and only if
$M(E\setminus\{a\})$ has a nontrivial projective summand.  

Our approach will be much based on studying the matroid $M$ via these localizations.
The localizations, matroids over $R_{\mf m}$, induce matroids over the residue fields $R_{\mf m}/(\mf m)$;
the generic matroid $\Mg$ is also over a field.  Both constructions thus give rise
to ordinary matroids (as the decomposition in \ref{lem:essential} is unique).


\subsection{Towards generalizations}\label{tg}
Classical matroids can be defined in several, ``cryptomorphic'' ways, for instance by giving axioms satisfied by the bases, or circuits, or independent sets. As we have seen, our definition of matroids over rings generalizes
the (co)rank function definition of matroids.  This makes the following question natural.
\begin{question}
Are there cryptomorphic definitions of matroids over rings?
\end{question}
Certain axiomatizations show some promise: for example, 
the discussion before Remark~\ref{refereequestion}
suggests that an axiomatization of polytopes of matroids over a DVR may be approachable.

The focus of this section is instead on the potential to generalize
the ``base'' of the construction, away from rings.  
With an eye to this, we will recast the axioms in a more categorical fashion, 
without reference to elements.
We also discuss polymatroids.
To begin, matroids over rings may be taken to be instead over affine varieties:

\begin{example}
Let $X$ be an affine algebraic variety, and $R=\mathcal O_X$ be the ring of regular functions on $X$. For every point of $X$, the functions vanishing at it form a maximal ideal of $R$. Then a matroid over $R$ can be seen as a \emph{bundle of matroids} over $X$.
Of course, since many of our general results hold when $R$ is a Dedekind ring, the first case to be investigated is when $X$ is an affine algebraic curve. 
In this case, by Propositions \ref{p:Dedekind M1} and~\ref{p:Dedekind M2}, $M$ is a matroid over $\mathcal O_X$ if and only if for every point of $X$ the corresponding localization $\Mm$ is a matroid over ${\mathcal O_X}_{\mf m}$,  $M_{gen}$ is a matroid over $\Frac(\mathcal{O}_X)$, and it is verified the global condition the Picard group $\Pic(X)$ stated in Proposition \ref{p:Dedekind M1}.
\end{example}

There appears to be no obstruction to patching these 
``bundles of matroids''
in a sheaf-theoretic fashion to yield (bundles of) matroids
over arbitrary schemes.  Proper investigation of
these is left to future work.

The next definition is Definition~\ref{def:matroid}
recast without reference to elements, with an eye towards
possible categorical generalizations; however,
reference is still made to cyclic modules.  
We have also separated out the statement (M1).

\begin{definition}
Let $R$ be a commutative
ring.  A \emph{matroid over $R$} on the ground set $E$
is a function $M$ assigning to each subset $A\subseteq E$
a finitely-generated $R$-module $M(A)$ 
satisfying the following axioms:
\begin{enumerate}\renewcommand{\labelenumi}{(M\arabic{enumi})}
\item For any $A\subseteq E$ and $b\in E\setminus A$, there exists a surjection
$M(A)\twoheadrightarrow M(A\cup\{b\})$ whose kernel is a cyclic submodule of $M(A)$.
\begin{samepage}
\item For any $A\subseteq E$ and $b,c\in E$, there exists a pushout
\[\xymatrix{
\ar@{}[dr]|{\mbox{\LARGE$\lrcorner$}}
M(A    )\ar[r]\ar[d] & M(A    \cup\{b\})\ar[d] \\
M(A    \cup\{c\})\ar[r] & M(A    \cup\{b,c\})
}\]
where all four morphisms are surjections with cyclic kernel.
\end{samepage}
\end{enumerate}
\end{definition}

We have already seen the pushout in the proof of Proposition~\ref{p:base change}.
Conversely, using the element-based criterion for pushouts, the fact that
\[\xymatrix{
M(A)\ar^\varphi[r]\ar_\psi[d] & M(Ab)\ar^{\psi'}[d] \\
M(Ac)\ar_{\varphi'}[r] & M(Abc)
},\]
is a pushout diagram of modules can be restated as
$$M(Abc)\simeq\frac{M(Ac)\oplus M(Ab)}{  \left\{  \big(\psi (x), -\varphi (x)\big), x\in M(A) \right\} }.$$
The fact that the maps are surjections implies
$$M(Abc)\simeq M(A) / (\ker \varphi, \ker \psi).$$
where by $(\ker \varphi, \ker \psi)$ we denote the submodule of $M(A)$ generated by the two kernels.
Then the elements $x$ and $y$ required by axiom~(M) can be chosen as generators
of $\ker\varphi$ and $\ker\psi$.

Realizability may also be recast; notionally, a realizable matroid
is still one in which the choices in axioms (M1) and~(M2) may be made globally.  
Let $\mc B(E)$ be the category of
the Boolean poset of subsets of~$E$, where inclusions of sets are the morphisms.

\begin{definition}
A matroid $M$ over~$R$ is \emph{realizable} if it is
the map on objects of some functor $F:\mc B(E)\to\fgMod{R}$,
and axioms~(M1) and~(M2) are satisfied by choosing the morphisms $F(A\to Ab)$.
A \emph{realization} of~$M$ is a choice of such an~$F$.
\end{definition}

Indeed, if a matroid $M$ over~$R$ is realizable in the above sense, 
corresponding to the functor $F$, 
then it is also realizable as defined before: 
the matroid $M_X$ of a vector configuration $(N,X=\{x_a\})$,
where $N$ is $F(\emptyset)$, and $x_a$ is a generator
of $\ker F(\emptyset\to\{a\})$ for each $a\in E$. 
Indeed, in this above setting, the pushout axiom~(M2) applied to $F$ guarantees that
equation~\eqref{eq:matroid from vectors} holds for all $A\subseteq E$.
The converse is similarly easy.

If we remove the cyclicity requirement, we wind up with polymatroids.

\begin{definition}
A \emph{polymatroid over $R$} on the ground set $E$
is a function $M$ assigning to each subset $A\subseteq E$
a finitely-generated $R$-module $M(A)$
such that
\begin{enumerate}
\item[(PM)] for every subset $A\subseteq E$ and elements
$b,c\in E$, there exist submodules
$K$ and $L$ of $A$ satisfying
\begin{align*}
M(A\cup\{b\}) &\cong M(A)/K \\
M(A\cup\{c\}) &\cong M(A)/L \\ 
M(A\cup\{b,c\}) &\cong M(A)/(K+L).
\end{align*}
\end{enumerate}
\end{definition}

Axiom (PM) is also equlvalent to axiom (M2) with the words ```with cyclic kernel'' stricken.
We note that Proposition~\ref{p:classical}
is true for polymatroids as well, and its proof goes through \emph{mutatis mutandis},
when the corank axiom (C1) is replaced with
\begin{enumerate}
\item[(PC)] For $A\subseteq E$ and $b\in E\setminus A$, $\cork(A)\geq\cork(Ab)$.
\end{enumerate}

\begin{example}\label{ex:polymatroids and matroids}
Not every polymatroid over~$R$ satisfying (M1) is a matroid over~$R$.
For a counterexample, let $R=\Z$.  There is a pushout diagram of surjections
\begin{equation}\label{eq:M2'}
\xymatrix{
\ar@{}[dr]|{\mbox{\LARGE$\lrcorner$}}
\Z\oplus\Z/2\Z\ar[r]\ar[d] & \Z/2\Z\ar[d] \\
\Z/4\Z\ar[r] & \Z/2\Z
}
\end{equation}
in which the top map has kernel $\<(2,0), (0,1)\>$
and the left map has kernel $\<(2,1)\>$.
Moreover there exist surjections $\Z\oplus\Z/2\Z\twoheadrightarrow\Z/2\Z$
with cyclic kernel: there are two such,
one with kernel $\<(1,0)\>$ and one with kernel $\<(1,1)\>$.
However, neither of these maps can be fitted into a pushout diagram
of surjections with groups isomorphic to~\eqref{eq:M2'}; both those
pushouts are the trivial group.
So diagram~\eqref{eq:M2'} corresponds to a function from $\mc B(2)$ to $\fgMod{\Z}$
that satisfies (PM) and (M1) but not (M2).
\end{example}

\begin{question}
There are various ways to axiomatize matroids using rank functions.
It is possible to state the submodularity axiom for matroids as our axiom (C2)
parallelling (M2), that is
\[\rank(Ab) + \rank(Ac) \geq \rank(A) + \rank(Abc),\]
but the more usual statement of this axiom doesn't
restrict to covers in the Boolean lattice: it asserts that
\[\rank(A) + \rank(B) \geq \rank(A\cap B) + \rank(A\cup B)\]
for all $A,B\subseteq E$.  Similarly, the fact that rank
is nondecreasing and bounded by cardinality can be framed
on covers, like our (M1), or on all containments.
\emph{Is there an axiomatization of matroids over~$R$
which replaces (M1) and~(M2) with axioms on 
all containments, respectively pairs of sets?}

Example~\ref{ex:polymatroids and matroids} suggests that such an axiom
system would still need to make reference to the number of generators of kernels.
It is also conceivable that the axioms sought would only agree over 
Dedekind domains: the behaviour exhibited in Example~\ref{ex:2D not combinatorial} below
for a non-Dedekind domain
interferes with na\"ive attempts to patch pushout squares together.
\end{question}

\section{Dedekind domains}\label{sec:Dedekind}
The theory arising from Definition~\ref{def:matroid} 
makes a good parallel
to the theory of classical matroids when $R$ is a Dedekind domain,
and this is the case we will give most attention to in the following
sections.
In this section we review some properties of Dedekind domains for use in
following sections.

One well-behaved feature of Dedekind domains in our setting is  Lemma~\ref{lem:Dedekind kernels}.
Example~\ref{ex:2D not combinatorial} shows that this fails in the two-dimensional setting.

\begin{lemma}\label{lem:Dedekind kernels}
Let $R$ be a Dedekind domain.
Given two $R$-modules $N$ and $N'$, all cyclic modules
that appear as kernels of surjections $N\twoheadrightarrow N'$
are isomorphic.
\end{lemma}

\begin{proof}[Proof of Lemma~\ref{lem:Dedekind kernels}]
Suppose we have two surjections
$N\twoheadrightarrow N'$ with kernels respectively
generated by elements $x$ and $y$ of~$N$.
We show that $\<x\>$ and $\<y\>$
are isomorphic as $R$-modules with the isomorphism given by $x\mapsto y$.
It is enough to show that this map is an isomorphism after
localizing at every maximal prime $\mf m$ of~$R$.
Now, the isomorphism class of $\<x\>_{\mf m}$ can
be read off of the other two modules in the localized exact sequence
\[0\to\<x\>_{\mf m}\to N_{\mf m}\to N'_{\mf m}\to 0.\]
To be precise, if the rank of $N_{\mf m}$ exceeds that of $N'_{\mf m}$,
then $\<x\>_{\mf m}\cong R_{\mf m}$ is free; otherwise,
$\<x\>_{\mf m}$ is torsion and is determined up to isomorphism
by its $(R/\mf m)$-dimension,
which is the difference of the dimensions of the torsion parts
of $N_{\mf m}$ and~$N'_{\mf m}$.
The isomorphism class of $\<y\>_{\mf m}$ is determined
in the same way from the same data, so that $\<x\>_{\mf m}\cong\<y\>_{\mf m}$
for all~$\mf m$.  And since $x$ and $y$ are generators,
the isomorphism can be taken to send $x\mapsto y$.
\end{proof}

\begin{example}\label{ex:2D not combinatorial}
Let $R = \K[x,y]/\<x,y\>^2$, the ring of two-dimensional first-order jets
(which is imprecisely the ``smallest'' two-dimensional ring).
Let $N$ be the length~3 $R$-module $\<x,y\>/\<x^2,y^2\>$,
where these $x$ and $y$ should be read as elements not of~$R$ but of~$\K[x,y]$:
thus $N$ is isomorphic to the so-called Matlis dual of $R$.
Then the quotients $N/\<x\>$ and $N/\<y\>$ are both isomorphic to $\K$,
but their kernels $\<x\>/\<x^2,xy^2\>$ and $\<y\>/\<x^2y,y^2\>$ are not isomorphic.
\end{example}

We next recall some structural results about $R$-modules.
Given an $R$-module $N$, let $\tors{N}\subseteq N$ denote the submodule
of its torsion elements, and $\projpart{N}$ denote 
the projective module $N/\tors{N}$. Then $N$ can be described as follows.

\begin{proposition}\label{p:Dedekind structure} \cite[exercises~19.4--6]{Eisenbud}
Let $R$ be a Dedekind domain. Every f.g.\ $R$-module $N$ is the direct sum of its torsion submodule $\tors{N}$ and of a
projective module isomorphic to $\projpart{N}$.

Every torsion module may be written uniquely up to isomorphism
as a sum of submodules $R/\mf m^k$ for $\mf m$ a maximal prime of~$R$ and $k\in\Z_{>0}$.
It also may be written uniquely as a sum of submodules
$R/I_1\oplus\cdots\oplus R/I_m$ (its invariant factors)
for a chain $I_1\subseteq\cdots\subseteq I_m$ of ideals of~$R$.

Every nonzero projective module
is uniquely isomorphic to $R^h\oplus I$
for some $h\geq 0$ and nonzero ideal~$I$, up to differing isomorphic choices of~$I$.
In particular, for ideals $I$ and $J$, we have $I\oplus J \cong R \oplus (I\otimes J)$, so that the direct sum of any two projective modules of ranks $i, j\geq 0$ is given by
\begin{equation}\label{sum-proj}
(R^{i-1}\oplus I)\oplus (R^{j-1}\oplus J) \simeq R^{i+j-1}\oplus (I\otimes J).
\end{equation}

\end{proposition}

We recall the following definitions. The \emph{Picard group} of $R$, $\Pic(R)$, 
is the group of fractional ideals of~$R$ under multiplication, modulo the subgroup of principal ideals.
If $R$ is Dedekind, then $\Pic(R)$ is isomorphic to the group of the isomorphism classes of f.g.\ projective modules of rank 1, with product induced by the tensor product. 
If $P$ is a projective module of rank $n$, the exterior algebra $\Lambda^n P$ is a f.g.\ projective module of rank ${n \choose n}=1$. We call \emph{determinant}, and denote by $\det(P)$, its class in $\Pic(R)$.

\begin{example}\label{ex:Pic}
The Picard group is trivial in many familiar Dedekind rings, 
including the cases that this paper works out in detail: $\Z$ and discrete valuation rings.
So we name here some examples where
it is not.  Among rings of integers of number fields we have examples
like $R = \Z[\sqrt{-5}]$, whose Picard group is isomorphic to $\Z/2\Z$, 
with the non-identity class necessarily represented by any non-principal ideal, for instance
$(2,1+\sqrt{-5})$.

Another example, among rings of regular functions, is the coordinate ring
$R=\mathbb C[x,y]/(y^2 - x(x-1)(x-\lambda))$ of an elliptic curve punctured
at the identity, whose Picard group is isomorphic to the group of points
of the (complete) elliptic curve as an abelian variety, that is
to the additive group $\C$ modulo an embedded lattice $\Z^2$:
the ideal $(x-x_0,y-y_0)$ of the point $(x_0,y_0)$ represents its class in $\Pic(R)$.
\end{example}

We will also find useful a description of the \emph{algebraic $K$-theory group} $K_0(R)$
of f.g.\ $R$-modules: that is, the abelian group
generated by classes $[N]$ of f.g.\ $R$-modules, modulo the relations 
$[N]=[N']+[N'']$ for any exact sequence 
\[0\to N'\to N\to N''\to 0.\]
\begin{proposition}\label{p:Dedekind K}
There is an isomorphism of groups
$$\Phi: K_0(R)\longrightarrow\Z\oplus\Pic(R).$$
\end{proposition}

\begin{proof}
In Corollary~2.6.3 of~\cite{Weibel} 
the $K$-theory group $K^0(R)$ 
of \emph{projective} $R$-modules is shown to be $\Z\oplus\Pic(R)$, 
via the map $$[P] \mapsto (\rk(P), \det(P))$$ as a consequence of Formula
\ref{sum-proj}.
But since $R$ is a regular ring, the natural homomorphism
$K^0(R)\to K_0(R)$ is an isomorphism \cite[\S15.1]{Fulton}.
\end{proof}

In virtue of the isomorphism above, from now on we will denote by $\det(N)$, 
and name \emph{determinant},
the class of any f.g.\ R-module $N$ in the Picard group, i.e. the second summand of $\Phi(N)$. In the same way, by $\rk(N)$ we denote the first summand of $\Phi(N)$: this coincides with the rank of 
$\projpart{N}$, i.e. with
the dimension of $N\otimes\Frac(R)$.

Note in particular that $\Phi$ extends the usual map from invertible ideals to $\Pic(R)$.

The potential nontriviality of this summand $\Pic(R)\subseteq K_0(R)$
has global consequences for matroids over~$R$: see Proposition~\ref{p:cl(M)} below.

\section{Duality}\label{sec:duality}
One of the first notions to be demanded of a putative generalization
of matroids is duality.  
Our construction of duality springs from the
case of realized matroids, where we have Gale duality of vector
configurations.  Some conditions are required on $R$
for this construction to produce a unique dual for any 
matroid over $R$.  One sufficient condition will be that
$R$ is a Dedekind domain, and therefore of global dimension~1;
this case will continue being our primary focus.
An earlier version of this paper asserted the existence of unique duals over local rings of any dimension,
but the proof was flawed.
We do not have a definitive answer to the following 
natural question:

\begin{question}
What is the most general class of rings $R$ for which 
duality of matroids over $R$ is well defined and correctly behaved?  
\end{question}

\addtocounter{theorem}{1}

We give an outline of this section here to provide
some guideposts for readers less familiar with homological algebra.
Duality is defined in Definition~\ref{def:dual};
its relationship to the Gale dual is Proposition~\ref{p:dual realizable}.  
The construction ultimately reduces to 
dualizing a map of modules, in the sense of applying
the functor $\Hom(\mbox{---},R)$.  We produce the needed map by
specifying a maximal chain of sets and composing the
quotient maps provided by axiom (M); 
it is of course necessary to check that this independent of the choices made
(Lemma~\ref{p:dual well-defined}), 
and that the resulting dual satisfies (M) as well (Theorem~\ref{p:dual}).
In the essential case, duality is an involution (Proposition~\ref{p:duality})
which behaves as expected under direct sums and minors (Proposition~\ref{p:duality2}).
Our most concrete description of the dual 
in the Dedekind case is Corollary~\ref{cor:dual}:
as we explain after the statement, this is a generalization of the duality formula 
for corank functions of usual matroids.  We also show in this case that
certain base changes we will use later behave well under duality 
(Proposition~\ref{p:duality and base change}).

Let $R$ be a Dedekind domain.
Let $M$ be a matroid over a ring $R$, on ground set~$E$.
For any $A\subseteq E$ and $b\in E\setminus A$,
the map provided by condition~(M1) may be fitted into an exact sequence
of the shape
\begin{equation}\label{eq:duality res 1}
0\to I \to R \to M(A)\to M(Ab)\to 0
\end{equation}
where $R/I$ is chosen isomorphic to the cyclic kernel from condition~(M1),
with $I$ an ideal of $R$.

The next ingredient is a projective resolution of form
\begin{align}\label{eq:duality res 2}
\ldots\to P_2^\emptyset\to P_1^\emptyset\to P_0^\emptyset\to M(\emptyset)\to 0,
\end{align}
where $P_0^\emptyset$ and $P_1^\emptyset$ are projective; this can
be attained from a projective resolution of $M$.  
Because $R$ is Dedekind, there is a
f.g.\ projective resolution of $M(\emptyset)$ of length at most~1,
so $P_2^\emptyset$ and the terms left of it are zero; 
fix one of these resolutions.  

From any maximal flag of subsets
$\emptyset=A_0\subsetneq A_1\subsetneq\cdots\subsetneq A_{|A|}=A$
we obtain a composite map
\[P_0^\emptyset\to M(\emptyset)\to M(A_1)\to\cdots\to M(A).\]
The kernel of this composition $P_0^\emptyset\to M(A)$ has a filtration whose subquotients
are the kernels of the individual arrows in it.  
The exact sequences above continue to free resolutions for these kernels,
allowing us to resolve $\ker P_0^\emptyset\to M(A)$ with a correspondingly filtered
resolution by the Horseshoe Lemma:
\[P(A)_\bullet :\quad \ldots\to P_2\to P_1^\emptyset\oplus R^{|A|}\stackrel{d_1}
\to P_0^\emptyset\to M(A)\to 0\]
The subquotient complexes appearing in the filtration
of the map $P_2\to P_1^\emptyset\oplus R^{|A|}$ in this complex
are one copy of $P_2^\emptyset\to P_1^\emptyset$, from \eqref{eq:duality res 2},
and $|A|$ copies of complexes $I\to R$, from \eqref{eq:duality res 1}.


For convenience we will give the modules in $P(A)_\bullet$ simpler names:
\[P(A)_\bullet :\quad\cdots\to P_2(A)\to P_1(A)\stackrel{d_1}
\to P_0(A)\to M(A)\to 0.\]
As usual, we write $^\vee$ for the contravariant functor $\Hom(\mbox{---},R)$.
\begin{definition}\label{def:dual}
Define the module $M^*(E\setminus A)$ as the cokernel of the map
dual to $d_1$ in $P(A)_\bullet$, that is
\[M^*(E\setminus A) \doteq \coker\big(P_0(A)^\vee\stackrel{d_1^\vee}
\longrightarrow  P_1(A)^\vee\big).\]
We define $M^*$, the \emph{dual} matroid over~$R$ to~$M$, to be the collection of
these modules $M^*(E\setminus A)$.
\end{definition}

\begin{lemma}\label{p:dual well-defined}
If $R$ is a Dedekind domain, then
the module $M^*(E\setminus A)$ is well-defined.
\end{lemma}

\begin{lemma}\label{lem:dual is Ext}
Let $R$ be a Dedekind domain.
For any exact sequence
\[0\to K_2\to Q_1\to Q_0\to N\to 0\]
of $R$-modules with $Q_1$ and $Q_0$ projective, 
the cokernel of the induced map $Q_0^\vee\to Q_1^\vee$
is isomorphic to $\Ext^1(N,R)\oplus\Hom(K_2,R)$.
\end{lemma}

\begin{proof}
Let $K_1$ be the kernel of $Q_0\to N$.
This splits the given sequence into two short exact sequences
\begin{gather*}
0\to K_2\to Q_1\to K_1\to 0\\
0\to K_1\to Q_0\to N\to 0
\end{gather*}
which yield the following long exact sequences of $\Ext(\mbox{---},R)$:
\begin{gather*}
0\to\Hom(K_1,R)\to\Hom(Q_1,R)\to\Hom(K_2,R)\to\Ext^1(K_1,R)\to0\\
0\to\Hom(N,R)\to\Hom(Q_0,R)\to\Hom(K_1,R)\to\Ext^1(N,R)\to0\to\Ext^1(K_1,R)\to0.
\end{gather*}
The last zero arises since $R$ has global dimension~1,
and it implies $\Ext^1(K_1,R)=0$.
The cokernel of the composition $\Hom(Q_0,R)\to\Hom(K_1,R)\to\Hom(Q_1,R)$
is canonically isomorphic to an extension of the cokernels of the
maps being composed, which is an extension of
$\Ext^1(N,R)$ by $\Hom(K_2,R)$.
The latter is projective, so the extension can (noncanonically)
be taken to be a direct sum.
\end{proof}

\begin{proof}[Proof of Lemma~\ref{p:dual well-defined}]
First of all, Lemma~\ref{lem:Dedekind kernels} implies that,
given a fixed maximal flag of subsets $\{A_i\}$ of~$A$, there is a unique choice
of the modules $I$ in each instance of \eqref{eq:duality res 1},
up to isomorphism.  
Therefore the isomorphism class of $P_2(A)$ 
is well-defined for each fixed flag.  

We are done so long as every maximal flag of
subsets yields the same projective module $P_2$.
One can obtain any maximal flag of subsets from any other
by successive replacements
of a segment $A_i\subsetneq A_ib\subseteq A_ibc$
with $A_i\subsetneq A_ic\subseteq A_ibc$,
so it's sufficient to show that one such replacement doesn't alter $P_2(A)$.
For any such replacement, there exists a commutative diagram as in axiom~(M).
\[\xymatrix{
M(A_i)\ar^f[r]\ar_g[d] & M(A_ib)\ar^{g'}[d] \\
M(A_ic)\ar_{f'}[r] & M(A_ibc)
},\]
Whichever of the two flags of subsets is used, these two maps correspond
to two steps like \eqref{eq:duality res 1} in the filtration of $P(A)_\bullet$.
In either case, the subquotient complex of $P(A)_\bullet$
formed from the extension formed of these two steps is
a resolution of $\ker(M(A_i)\to M(A_ibc))$ like
\[0\to K\to R^2\stackrel{d}\to
\ker\Big(M(A_i)\to M(A_ibc)\Big)\to 0\]
where the labelled map $d$ may be chosen to be $(r,s)\mapsto rx+sy$
if $x$ and~$y$ generate the kernels of $f$ and~$g$, respectively.
It follows that $K$, and therefore the module $P_2(A)$, is isomorphic in the two cases.

Finally, by Lemma~\ref{lem:dual is Ext}, since $M^*(E\setminus A)$ depends only
on the isomorphism classes of $P_2(A)$ and $M(A)$ itself,
it is well-defined.
\end{proof}

\begin{theorem}\label{p:dual}
If $M$ is a matroid over~$R$, and $M^*$ is defined,
then $M^*$ is a matroid over~$R$ as well.
\end{theorem}

\begin{proof} 
Let $A\subseteq E$ and $b\in E\setminus A$.
In the construction of $P(Ab)_\bullet$, choose
a maximal flag of subsets ending in $\cdots\subseteq A\subseteq Ab$.
The construction then provides an exact sequence of complexes
which, at the $P_0$ and~$P_1$ terms, looks like
\[\xymatrix{
0\ar[r] & P_1(A)\ar[r]\ar[d] & P_1(Ab)\ar[r]\ar[d] & R\ar[r]\ar[d] & 0 \\
0\ar[r] & P_0(A)\ar[r] & P_0(Ab)\ar[r] & 0 &
}.\]
All these modules are projective, so dualizing all the maps preserves exactness:
we have
\begin{equation}\label{eq:dual maps}
\xymatrix{
0 & P_1(A)^\vee\ar[l] & P_1(Ab)^\vee\ar[l] & R^\vee\ar[l] & 0\ar[l] \\
0 & P_0(A)^\vee\ar[l]\ar[u] & P_0(Ab)^\vee\ar[l]\ar[u] & 0\ar[l]\ar[u] &
}
\end{equation}
This induces a map between the cokernels of the left two upward arrows,
which is still surjective, and has kernel some quotient of $R$.
That is, we have a surjection
$M^*(E\setminus A)\gets M^*(E\setminus(Ab))$ whose kernel
is a cyclic module.  These maps are exactly what is needed
to establish condition~(M1) for~$M^*$.

Now let $b,c\in E\setminus A$.
Building off the maps in the pushout diagram assured by axiom~(M) for~$M$,
we get a commuting square of the maps among the modules $P_1$
constructed above.
\[\xymatrix{
P_1(A)\ar[r]\ar[d] & P_1(Ab)\ar[d] \\
P_1(Ac)\ar[r] & P_1(Abc) \\
}\]
Each of these inclusions has cokernel $R$, and so the target splits
as a direct sum.  Regard the various complexes $P(\cdot)_\bullet$
as resolutions of kernels $\ker(P_0^\emptyset\to M(\cdot))$.
Then, taking for example the top map,
$P_1(A)\to P_1(Ab)$, we can identify
$P_1(Ab)$ with $P_1(A)\oplus R$, where
$P_1(A)$ maps to $\ker(P_0^\emptyset\to M(Ab))$
via its map to $\ker(P_0^\emptyset\to M(A))$, and
$R$ maps to $\ker(P_0^\emptyset\to M(Ab))$ by sending
1 to a lift of a generator of $\ker(M(A)\to M(Ab))$.

Now, this lift of a generator of $\ker(M(A)\to M(Ab))$
to $P_0$ is also a lift of a generator of $\ker(M(Ac)\to M(Abc))$.
The same is true with the roles of $b$ and $c$ reversed.  So in fact
the whole square of maps can be split compatibly, as
\[\xymatrix{
P_1(A)\ar[r]\ar[d] & P_1(A)\oplus R\ar[d]^{(x,r)\mapsto (x,r,0)} \\
P_1(A)\oplus R\ar[r]_{\strut (x,r)\mapsto (x,0,r)} & P_1(A)\oplus R^2 \\
}.\]
Dualizing this square yields the square
\begin{equation}\label{eq:dual proof 3}\xymatrix{
P_1(A)^\vee & P_1(A)^\vee\oplus R^\vee\ar[l] \\
P_1(A)^\vee\oplus R^\vee\ar[u] & P_1(A)\oplus (R^\vee)^2\ar[l]\ar[u] \\
}\end{equation}
in which all the maps are projections onto summands, which is a pushout.

Finally, the square with which we are ultimately concerned
\begin{equation}\label{eq:dual proof 4}\xymatrix{
M^*(E\setminus A) & M^*(E\setminus(Ab))\ar[l] \\
M^*(E\setminus(Ac))\ar[u] & M^*(E\setminus(Abc))\ar[l]\ar[u] \\
}\end{equation}
is obtained by taking the quotient of each of the modules in square~\eqref{eq:dual proof 3}
by the image of the corresponding module $P_0(\cdot)^{\vee}$.  In fact all the
$P_0(\cdot)^{\vee}$ are isomorphic to $(P_0^\emptyset)^\vee$, compatibly.
This remains a pushout by the universal property, as follows.
Commuting maps from $M^*(E\setminus(Ab))$ and
$M^*(E\setminus(Ac))$ to a module $N$ lift to commuting maps to~$N$
from the upper-right and lower-left instances
of $P_1(A)^\vee\oplus R^\vee$ in \eqref{eq:dual proof 3}, whose kernels contain $(P_0^\emptyset)^\vee$.
Since that square is a pushout, a map $P_1(A)^\vee\to N$ can be provided.
The kernel of this map contains $(P_0^\emptyset)^\vee$ and so it descends to a map
$M^*(E\setminus A)\to N$.  Uniqueness can be argued similarly.
We have thus established axiom~(M) for $M^*$.
\end{proof}

We now state a fact whose substantive content is 
Lemma~\ref{lem:dual is Ext}.
We have postponed it to here only so that ``matroid over~$R$''
could appear in the statement.

\begin{corollary}\label{cor:dual essential}
Let $R$ be a Dedekind domain. Then $M^*$ is an essential matroid over~$R$. 

Suppose a system of modules $\{M'(A) : A\subseteq E\}$ is constructed as in
Definition~\ref{def:dual} except that we allow sequence \eqref{eq:duality res 2}
to be an arbitrary projective resolution.  Then $M'(A)$ is the direct sum
of $M^*(A)$ and a projective empty matroid.  
\end{corollary}

\begin{proof}
The module $P_2(\emptyset)$ is trivial, and therefore
$M^*(E)\cong \tors{M(\emptyset)}$ by Lemma~\ref{lem:dual is Ext}.

If now we substitute for \eqref{eq:duality res 2} 
a different projective resolution, say with second syzygy module $K_2$,
the effect on the sequence $P(A)_\bullet$ is to add
a $K_2$ summand to the cokernel of the dual of
the differential $d_1$.  So by Lemma~\ref{lem:dual is Ext} again,
$M'(A)$ will differ from
\[M^*(E\setminus A) \cong \Ext^1(M(A),R)\oplus\Hom(P_2(A),R)
\cong \tors{M(A)}\oplus P_2(A)^{\vee}\]
only up to the projective summand $\Hom(K_2,R)$.
\end{proof}

Our notion of duality reduces to the classical Gale duality
of vector configurations in the realizable case.  It was prefigured by the construction of
the dual of a realizable arithmetic matroid in~\cite{D'Adderio-Moci},
in which the matrix transpose operation used to construct the Gale dual
corresponds to dualizing our differential~$d_1$.

\begin{proposition}\label{p:dual realizable}
If a matroid $M$ is realizable
and its dual $M^*$ is defined, then $M^*$ is realizable too.
\end{proposition}

\begin{proof}
Let $M$ be realized by the vector configuration $(x_a:a\in E)$
within $M(\emptyset)$.  Fix lifts of these vectors to
vectors $(\tilde x_a:a\in E)$ within $P^\emptyset_0$.

For a set $A\subseteq E$ and an element $a\in A$, 
the map $d_1^A:P_1^\emptyset\oplus R^A\to P_0^\emptyset$
appearing in the resolution $P(A)_\bullet$
satisfies $d_1(0,e_a) = \tilde x_a$.  Thus each of these
maps is a restriction of the map $d_1^E:P_1^\emptyset\oplus R^E\to P_0^\emptyset$
in the complex $P(E)_\bullet$ to the submodule
$P_1^\emptyset\oplus R^A\subseteq P_1^\emptyset\oplus R^E$.
Let us dualize, and write $\{e^a : a\in E\}$ for the dual standard basis of $(R^E)^\vee$.
The map $(d_1^A)^\vee$ factors as $q_A\circ(d_1^E)^\vee$,
where $q_A:(P_1^\emptyset\oplus R^E)^\vee\to(P_1^\emptyset\oplus R^A)^\vee$
is the quotient map by the submodule $\<e^a : A\in E\setminus A\>$.  
Hence, $M^*(E\setminus A) = \coker((d_1^A)^\vee)$ is the quotient
of $M^*(\emptyset) = \coker((d_1^E)^\vee)$ 
by the submodule generated by the images of the $e^a$ for $a\in E\setminus A$.
But we have now exactly described a vector configuration realizing $M^*$:
the ambient module is $M^*(\emptyset) = \coker((d_1^E)^\vee)$,
and the vector labelled by $a$ is the image of $e^a$.  
\end{proof}

Matroid duality over~$R$ has the properties expected of it.

\begin{proposition}\label{p:duality2}
If $M$ and $M'$ are matroids over a ring $R$ where matroid duality
is defined, then
\begin{enumerate}\renewcommand{\labelenumi}{(\alph{enumi})}
\item $(M\setminus i)^*=M^*/i$.
\item $(M/i)^*$ is isomorphic to $M^*\setminus i$ plus a projective empty matroid.
If $M(\{i\})$ is a quotient of $M(\emptyset)$ by $R$, then $(M/i)^*=M^*\setminus i$.
\item $(M\oplus M')^*=M^*\oplus M'^*$.  
\end{enumerate}
\end{proposition}

The extra hypothesis in part (b) plays the role of the condition that one
should not contract a loop of a classical matroid.  Part (a) has no corresponding
condition because all dual matroids are essential, which adjusts for the
discrepancy that might otherwise be expected if a coloop is deleted.

\begin{proof}
\noindent\emph{To (a).}  This is immediate: $M(A)$ equals $(M\setminus i)(A)$
for $A\not\ni i$, so the chains of maps used in constructing $P(A)_\bullet$
are identical in $M^*$ and $(M\setminus i)^*$.  In the former dual,
the dual of $d_1$ in $P(A)_\bullet$ is $M^*(E\setminus A) = (M^*/i)(E\setminus i\setminus A)$; in the latter it is $(M\setminus i)^*(E\setminus i\setminus A)$.

\noindent\emph{To (b).}  The matroid $M^*\setminus i$ over $R$ is obtained 
by the construction of Definition~\ref{def:dual} on the 
matroid $M/i$ except that
the resolution used of $(M/i)(\emptyset) = M(\{i\})$ 
is not the one specified there, but rather $P(\{i\})_\bullet$.
By Corollary~\ref{cor:dual essential}, $M^*\setminus i$ is the
direct sum of $(M/i)^*$ and a projective empty matroid,
so it is enough to show that $M^*\setminus i$ is essential,
i.e.\ that $M^*(E\setminus i)$ has no projective summands.

If $M^*(E\setminus i)$ had a nonzero projective summand $Q$, it would pull
back to a summand of $P_1(\{i\})^\vee$, and this can be done
in such a way that $0\to Q\to Q\to 0$ is a summand of
\[P_0(\{i\})^\vee\to P_1(\{i\})^\vee\to M^*(E\setminus i)\to 0.\]
In the undualized complex $P(\{i\})_\bullet$, this would appear
as a summand $Q^\vee\to 0$ of~$d_1$.   

However, in the construction of $P(\{i\})_\bullet$, 
the fact that the map $M(\emptyset)\twoheadrightarrow M(\{i\})$ has kernel $R$
implies that the images $d_1(P_1^\emptyset)$ and $d_1(R)$ intersect in zero.
Because $d_1|_{P_1^\emptyset}$ is part of a minimal resolution,
it has no direct summand lying in ite kernel.
Together, these implies that there is no nonzero summand $Q$ of $P_1(A)$ 
lying in $\ker d_1$, as was sought above.

\noindent\emph{To (c).}  To distinguish the complexes used
in the various matroids at hand, let us write the complex 
$P(A)_\bullet$ for the matroid $M$ with a subscript $P_M(A)_\bullet$
(and similarly for the other matroids involved).
For subsets $A$ and $A'$ of the respective
ground sets of $M$ and $M'$, it is easy to check from the definition 
that these complexes (can be taken to) split as
\[P_{M\oplus M'}(A\amalg A')_\bullet = P_{M}(A)_\bullet \oplus P_{M'}(A')_\bullet ,\]
from which the result follows.
\end{proof}

\begin{proposition}\label{p:duality}
If $M$ is a matroid over a ring $R$ for which duality is defined, then
$M$ is the direct sum of $M^{**}$ and a projective empty matroid.
In particular, if $M$ is essential, $M^{**}=M$.
\end{proposition}

\begin{proof}
Suppose first that $M$ is a realizable matroid over $R$, and fix
a realization $(x_a : a\in E)$.  The proof of Proposition~\ref{p:dual realizable}
gives a configuration of module elements realizing~$M^*$.  
Now, if $M^*(\emptyset)$ is given a projective resolution whose 
first map is $(d_1^E)^\vee$, and this resolution is used in place of 
\eqref{eq:duality res 2} to construct a system of modules ${M^*}'$, 
Corollary~\ref{cor:dual essential} shows that ${M^*}'$ is the direct
sum of $M^{**}$ and a projective empty matroid.

Let us write simply $d = d_1^\emptyset:P_1^\emptyset\to P_0^\emptyset$ for the first differential
in the resolution of $M(\emptyset)$.  Then the map $d_1^E$ is given
in matrix notation, treating direct sums as spaces of column vectors, by
\[d_1^E : P_1^\emptyset\oplus R^E 
\stackrel{\begin{pmatrix}d & x\end{pmatrix}}{-\!\!\!-\!\!\!\longrightarrow} 
P_0^\emptyset\]
where $x:R^E\to P_0^\emptyset$ is determined by $x(e_a) = \tilde x_a$, this $\tilde x_a$
being a lift of $x_a$ to $P_0^\emptyset$.
In the dual, the vector configuration which realizes $M^*$ are
the images of the standard basis of $(R^E)^\vee$.  So when
we run the duality construction the second time, the first differential
of the resolution $P(E)_\bullet$ in $M^*$ is
\[(P_0^\emptyset)^\vee\oplus R^E 
\stackrel{\begin{pmatrix}d^\vee & 0 \\ x^\vee & 1\end{pmatrix}}{-\!\!\!-\!\!\!\longrightarrow}
(P_1^\emptyset)^\vee\oplus (R^E)^\vee,\]
where the map $1:R^E\to(R^E)^\vee$ sends the primal to the dual standard basis.
Therefore, ${M^*}'$ is realized within the cokernel of
\[P_1^\emptyset\oplus R^E 
\stackrel{\begin{pmatrix}d & x \\ 0 & 1\end{pmatrix}}{-\!\!\!-\!\!\!\longrightarrow}
P_0^\emptyset\oplus (R^E)^\vee,\]
(the maps $1$ and $0$ being self-dual) by the image of the standard basis $\{(0,e^a):a\in E\}$
of the $(R^E)^\vee$ summand
on the right.  By a change of basis on the target of the above map,
corresponding to composition on the right by the isomorphism
\[P_0^\emptyset\oplus (R^E)^\vee
\stackrel{\begin{pmatrix}1&-x\\0&1\end{pmatrix}}{-\!\!\!-\!\!\!\longrightarrow}
P_0^\emptyset\oplus (R^E)^\vee,\]
we get that ${M^*}'$ is realized within the cokernel of
\[P_1^\emptyset\oplus R^E 
\stackrel{\begin{pmatrix}d & 0 \\ 0 & 1\end{pmatrix}}{-\!\!\!-\!\!\!\longrightarrow}
P_0^\emptyset\oplus (R^E)^\vee,\]
with the realizing vector configuration the image of $\{(\tilde x_a,e^a):a\in E\}$.  But the identity map
$R^E\to (R^E)^\vee$ that is a summand of the above map can be ignored in the cokernel;
that is, ${M^*}'$ is realized in $\coker d$ by $\{x_a:a\in E\}$, which shows ${M^*}'\cong M$.
We conclude that $M$ is the direct sum of $M^{**}$ and a projective empty matroid.

Now, let us drop the assumption that $M$ is realizable.  Axiom (M) indicates
that every two-element minor of $M$ is realizable, so by the above discussion
and Proposition~\ref{p:duality2}(a,b), the modules in the corresponding minor of $M^{**}$
differ from those of $M$ only up to projective summands.  
The rings over which we have defined matroid duality have the cancellative property
for projective summands: that is, if $N$, $N'$, $P$ are $R$-modules with $P$ projective
such that $N\oplus P\cong N'\oplus P$, then $N\cong N'$.  
So, since any two subsets $A,B\subseteq E$,
may be linked with a chain of overlapping two-element minors on which the 
realizable result above may be invoked,
we conclude there are fixed projective modules $P$, $Q$ such that
$M(A)\oplus P\cong M^{**}(A)\oplus Q$ for any $A\subseteq E$.
The proposition is proved by observing that, 
since $M^{**}$ is essential, $P$ must be zero; and if $M$ is essential as well,
so must $Q$ be zero.
\end{proof}

The remainder of this section is dedicated to explicit formulae for the modules
making up the dual matroid in the Dedekind case.

\begin{proposition}\label{p:cl(M)}
Let $M$ be a matroid over a Dedekind domain $R$.  The element of $\Pic(R)$
\begin{align*}
\cl(M) &\doteq \det(\projpart{M(A)}) + \det(\projpart{M^*(E\setminus A)}) + \det(\tors{M(A)})
\\&= \det(M(A)) + \det(M^*(E\setminus A)) - \det(\tors{M(A)})
\end{align*}
 is independent of the choice of $A\subseteq E$.
\end{proposition}

In particular, if $A$ is an independent set (or, more strongly, a basis) 
of the generic matroid of~$M$,
then by Corollary \ref{cor:dual essential} we have $\det(M(A)) = \cl(M)$.

\begin{proof}
Given a set $A\subseteq E$,
let $\cl(M)(A)$ be the value of $\cl(M)$ computed using that given
choice of~$A$.  It is enough to show that, for each $A\subseteq E$ and $b\in E\setminus A$,
$\cl(M)(A)$ equals $\cl(M)(Ab)$.

Given $A$ and $b$, it is true of exactly one of the two dual maps
$M(A)\to M(Ab)$ and $M^*(E\setminus(Ab))\to M^*(E\setminus A)$
provided by condition~(M1) that the rank of the target is one less
than the rank of the source.  
In the other map, these two ranks are equal.

If the first map has this rank drop, then its kernel must be isomorphic to~$R$,
and the exact sequence
\[0\to R\to M(A)\to M(Ab)\to 0\]
implies that $\det(M(A)) = \det(\projpart{M(A)}) + \det(\tors{M(A)})$ equals
$\det(M(Ab)) = \det(\projpart{M(Ab)}) + \det(\tors{M(Ab)})$.
Since the second map has no rank drop, its kernel is
contained in the torsion submodule of its source, so
$\det(\projpart{M(E\setminus A)})$ equals $\det(\projpart{M(E\setminus(Ab))})$.
Adding these equalities, we have $\cl(M)(A) = \cl(M)(Ab)$.  

If instead the second map has the rank drop, then 
the same argument shows that $\cl(M)(A) = \cl(M)(Ab)$
after exchanging $M$ for $M^*$ and sets $A$ for their complements $E\setminus A$,
and using the fact that $\tors{M(A)}\cong \tors{M^*(E\setminus A)}$.
\end{proof}

Recall that $K_0(R) = \Z\oplus\Pic(R)$.
Let $\sigma:K_0(R)\to K_0(R)$ be the involution acting as
the identity on the summand $\Z$ and negation on the summand $\Pic(R)$.

\begin{corollary}\label{cor:dual}
If $M$ is a matroid over a Dedekind domain $R$,
then $M^*(E\setminus A)$ is the module
whose torsion part is $\tors{M^*(E\setminus A)} \cong \tors{M(A)}$
and whose projective part is determined by the equality
\begin{equation}\label{eq:dual}
[\projpart{M^*(E\setminus A)}] = \sigma\big([M(A)] + |A|\cdot[R] - [M(\emptyset)]\big)
\end{equation}
in $K_0(R)$.
\end{corollary}
Note that, over a field, equation~\eqref{eq:dual} specializes to the formula for
dualizing rank functions familiar from the matroid setting,
\[\cork_M^*(E\setminus A) = \cork_M(A) + |A| - r\]
where $r$ is the rank of~$M$.

\begin{proof}
The assertion on the torsion parts is noted in the discussion after Lemma~\ref{lem:dual is Ext}.

As for the projective part, we treat the summands of $K_0(R) = \Z\oplus\Pic(R)$ separately.
In the $\Pic(R)$ summand, Proposition~\ref{p:cl(M)} implies that
\[\det(\projpart{M^*(E\setminus A)}) + \det(M(A)) = \cl(M) = \det(\projpart{M^*(E)}) + \det(M(\emptyset)),\]
which becomes the $\Pic(R)$ part of~\eqref{eq:dual} on noting that
$M^*$ is essential so $\projpart{M^*(E)}$ is trivial.

Regarding the rank, 
consider again the ideal $I\cong\coker(f^\vee)$ in the proof of Proposition~\ref{p:duality}.
Since $^\vee$ preserves the rank and $f^\vee$ is an injection, we have 
\[\rk(I) = \rk(M(A))-\rk (M(Ab)).\]
As well,
\[\rk (M^*(E\setminus Ab))-\rk (M^*(E\setminus A)) = \rk (R/I) = 1-\rk (I).\]
By induction on the size of $E\setminus A$, we get that 
$\rk (M(A)) - \rk (M^*(E\setminus A)) + |A|$ is constant, and thus
always equal to its value $\rank M(\emptyset)$ taken when $A=\emptyset$.
This proves the part of~\eqref{eq:dual} in the $\Z$ summand.
\end{proof}

\begin{proposition}\label{p:duality and base change}
Let $M$ be a matroid over a Dedekind domain $R$.
\begin{enumerate}\renewcommand{\labelenumi}{(\alph{enumi})}
\item Let $f:R\to S$ be a flat map to a Dedekind domain $S$.
Then $(M\otimes S)^*=M^*\otimes S$ (as matroids over~$S$).
\item Let $f:R\to S$ be the quotient by a maximal ideal.
Then, again, $(M\otimes S)^*=M^*\otimes S$.
\end{enumerate}
\end{proposition}

\begin{proof}
Since base changes 
are computed one module at a time, these are straightforward to check 
given Corollary~\ref{cor:dual}.  

For part~(a), to begin, we have $\tors{(M\otimes S)^*(A)} = \tors{(M(E\setminus A)\otimes S)}$
directly.  On the other hand, since projective modules remain projective
under $\mbox{---}\otimes S$, we find that $\tors{(M^*(A)\otimes S)}$
equals $\tors{(\tors{M^*(A)}\otimes S)}$, which in turn is
$\tors{(\tors{M(E\setminus A)}\otimes S)} = \tors{(M(E\setminus A)\otimes S)}$.
So the torsion parts agree.

As for the projective parts, because $f$ is flat,
the induced homomorphism $f_*:K_0(R)\to K_0(S)$
is given simply by $f_*[N] = [N\otimes S]$.  
Also, torsion modules remain torsion on tensoring with~$S$,
so that the operations $\mbox{---}\otimes S$ and $\projpart{\mbox{---}}$ commute.
Hence, using \eqref{eq:dual},
\[[\projpart{(M\otimes S)^*(A)}] = \sigma\big([M(A)\otimes S] + |A|\cdot[S] - [M(\emptyset)\otimes S]\big)\]
equals
\begin{multline*}
[\projpart{(M^*\otimes S)(A)}] = [\projpart{M^*(A)}\otimes S] = f_*[\projpart{M^*(A)}]
\\= f_*\,\sigma\big([M(A)] + |A|\cdot[R] - [M(\emptyset)]\big)].
\end{multline*}

For part~(b), $M\otimes S$ is a classical matroid, over a field,
and so we need only check that the corank functions of $(M\otimes S)^*$ 
and $M^*\otimes S$ are equal.  Let $I$ be the maximal ideal such that $S=R/I$.  
For a f.g.\ $R$-module $N$, the $S$-dimension of~$N\otimes S$ 
is the rank of $\projpart{N}$ plus $\dim_S\Tor_1^R(N,S)$;
the latter summand is the number of indecomposible
summands of $N$ isomorphic to $R/I^n$ for some~$n$.
Now,
\[\cork_{(M\otimes S)^*}(E\setminus A) = \cork_{M\otimes S}(A) + |A| - r\]
where $r$ is the generic rank of $M\otimes S$.
The term $\cork_{M\otimes S}(A)$ is computed as described just above, for $N = M(A)$:
we get $\rank M(A) + \dim_S\Tor_1^R(M(A),S)$.  
On the other hand, we know that $M^*(E\setminus A)$ has the same
projective part as $M(A)$; this means  $\Tor_1^R(M^*(E\setminus A),S) = \Tor_1^R(M(A),S)$.
And the rank of the module $M^*(E\setminus A)$
is $\rank M(A) + |A| - \rank M(\emptyset) = \rank M(A) + |A| - r$,
by \eqref{eq:dual}.  Therefore the dimension of $M^*(E\setminus A)$ is 
\[\rank M(A) + |A| - r + \dim_S\Tor_1^R(M^*(E\setminus A),S),\] 
which agrees with $\cork_{(M\otimes S)^*}(E\setminus A)$ as required.
\end{proof}

\section{Structure of matroids over a DVR}\label{sec:DVR}
In this section and the next
we record some structure theorems for matroids over~$R$
in terms of structure theorems for the modules over~$R$ themselves.
Our analysis of general Dedekind domains in the next section
will make much use of base changing to localizations of $R$,
so we begin here with the local case, i.e.\ where $R$ is a 
discrete valuation ring.  

We will see that these objects have connections to tropical geometry.
A matroid over a DVR (discrete valuation ring) $R$ defines a point on each Dressian, one of the
tropical analogues of the Grassmannian; 
this is equivalent to being a valuated matroid.  
(As per Remark~\ref{rem:pruefer}, ``discrete'' appears to be inessential here, 
so rings familar to tropicalists like the Puiseux series should also serve.)

For the whole of this section, $R$ will be a DVR with maximal ideal~$\mf m$.  
We first recall the structure theory of f.g.\ $R$-modules:
any indecomposible f.g.\ $R$-module
is isomorphic to either $R$ or $R/\mf m^n$ for some integer $n\geq 1$.
We will sometimes formally subsume $R$ into the latter family
by writing it as $R/\mf m^\infty$.  
So, if $N$ is a f.g.\ $R$-module and $i\geq 1$ is an integer,
define 
\[d_i(N) \doteq \dim_{R/\mf m}(\mf m^{i-1}N / \mf m^iN),\]
and
\[d_{\leq i}(N) \doteq \sum_{j=1}^i d_j(N) = \dim_{R/\mf m}(N / \mf m^iN),\]
and for convenience $d_i(N) = d_{\leq i}(N) = 0$ if $i\leq0$.
Let $d_\bullet(N)$ denote the infinite sequence of these.
We have
\[d_i(R/\mf m^n) = \begin{cases}
1 & 0<i\leq n \\
0 & i>n
\end{cases},\]
where $n$ may be $\infty$.
The following is a quick consequence.
\begin{proposition}\label{p:DVR M0}
Isomorphism types of f.g.\ $R$-modules are in
bijection with nonincreasing infinite sequences $d_\bullet$ of nonnegative integers
indexed by positive integers, 
the bijection being given by
\[N \longleftrightarrow d_\bullet(N).\]
\end{proposition}

This bijection permits a straightforward identification of those isomorphism classes
of modules which permit maps satisfying condition~(M1).
\begin{proposition}\label{p:DVR M1}
Let $N$ and $N'$ be f.g.\ $R$-modules.  There exists a surjection $\phi:N\to N'$
with cyclic kernel if and only if 
\[d_i(\phi) \doteq d_i(N) - d_i(N')\]
equals 0 or~1 for each $i\geq1$.
\end{proposition}

We can also easily extract the $d_i(\phi)$.

\begin{corollary}\label{cor:DVR M1}
Let $\{e_\alpha\}$ be a minimal set of generators for an f.g.\ $R$-module $N$,
and suppose $e_\alpha$ generates a summand isomorphic to $R/\mf m^{\ell_\alpha}$, 
wherein $\ell_\alpha$ may be $\infty$.
Let $x = \sum x_\alpha e_\alpha$ be an element of~$N$, and $\phi$ the canonical map $N\to N/\<x\>$.
Then $d_\bullet(\phi)$ is the lexicographically least sequence $d_\bullet$ 
such that for every $\alpha$,
\begin{equation}\label{eq:DVR d}
\#\{i\leq\ell_\alpha : d_i=0\} \leq \dim_{R/\mf m}(\<e_\alpha\>/\<x_\alpha e_\alpha\>).
\end{equation}
\end{corollary}

When $\ell_\alpha$ is finite, condition~\eqref{eq:DVR d} is equivalent to
\[d_{\leq \ell_\alpha} \geq \dim_{R/\mf m}(\<x_\alpha e_\alpha\>).\]

In the case that $N$ and $N'$ have finite length, Proposition~\ref{p:DVR M1}
follows from facts about the Hall algebra \cite{Macdonald}.
Indeed, it is equivalent that $N$ have finite length and that
$d_i(N)$ stabilize to 0 for $i\gg0$.  In this case $d_i$ is a partition,
and its conjugate partition is the one usually used to label $N$.  
For a cyclic module, this conjugate partition has a single row.
Then, under the specialization taking the Hall polynomials
to the Littlewood-Richardson coefficients, 
Proposition~\ref{p:DVR M1} is a consequence of the Pieri rule.
(Taking this further, our foundational Lemma~\ref{lem:Dedekind kernels}
is essentially the statement that all coefficients in the Pieri
rule are equal to~1.)

We include a proof of the proposition nonetheless, both because we do not require finite length
and because we reuse its framework in proving Corollary~\ref{cor:DVR M1}.

\begin{proof}
\noindent\emph{Necessity.}  Let $\<x\>$ be the cyclic kernel of $N\to N'$,
for $x\in N$.
The kernel of the induced surjection $N\otimes R/\mf m^n\to N'\otimes R/\mf m^n$
is 
\[K_n = \<x\>/(\<x\>\cap \mf m^nN).\]
The dimensions over $R/\mf m$ of these three modules are related by
\[d_{\leq n}(N) - d_{\leq n}(N') = \dim_{R/\mf m} K_n\]
and, by subtracting two such relations,
\[d_n(N) - d_n(N') = \dim_{R/\mf m} K_n - \dim_{R/\mf m} K_{n-1}.\]
It is clear by definition that the $K_n$ are an increasing sequence of modules,
so that $\dim_{R/\mf m} K_n - \dim_{R/\mf m} K_{n-1}$ is nonnegative.
On the other hand,
\[(\<x\>\cap \mf m^{n-1}N)/(\<x\>\cap \mf m^nN)\]
has length at most 1, since if $\mf m^ix\subseteq\mf m^{n-1}N$
then $\mf m^{i+1}x\subseteq\mf n^nN$.  
But this length is $\dim_{R/\mf m} K_n - \dim_{R/\mf m} K_{n-1}$,
which is thus at most~1.

\noindent\emph{Sufficiency.}  Given $N$ and an infinite list 
$\delta_i\in\{0,1\}$ such that $d_i(N)-\delta_i$ is also a nonincreasing 
sequence of naturals, equal therefore to $d_i(N')$ for a module $N'$,
we wish to construct $x\in N$ so that $N/\<x\>\cong N'$.

Let $I$ be the set of indices $i$ for which $\delta_i=1$ and $\delta_{i+1}=0$;
also include in~$I$ the symbol $\infty$ if $\delta_i=1$ for all sufficiently large~$i$.
For each $i\in I$, there is a summand isomorphic to $R/\mf m^i$ in~$N$.
Splitting off one module of each of these isomorphism classes, 
we can make the identification
\[N = \bigoplus_{i\in I} R/\mf m^i\oplus P\]
for some module~$P$, and let $e_i:i\in I$ be generators of the summands other than~$P$.
Let $t\in R$ be a generator of~$\mf m$, and define
\[x = \sum_{k\in I} t^{k-\delta_{\leq k}}\, e_k,\]
where as expected $\delta_{\leq k}$ means $\sum_{i=1}^k\delta_i$.  
(In fact the whole expression $k-\delta_{\leq k}$ must be interpreted as $\sum_{i=1}^k(1-\delta_i)$
if $k=\infty$.)

The module $P$ will remain as a summand in~$N/\<x\>$, and 
we may restrict attention to the remaining summand, call it $Q$.  
Towards describing it, define the elements 
\[\tilde e_i = \sum_{k\in I, k\geq i} t^{(k-\delta_{\leq k})-(i-\delta_{\leq i})}\, e_k\quad\in N.\]
Fix for the moment some $i\in I$.
Let $j=j(i)$ be the greatest index less than~$i$ such that 
$\delta_j=0$ and $\delta_{j+1}=1$, or if there is no such index let $j=0$.
Then we have 
\[t^j\tilde e_i = t^{\delta_{\leq j}}x.\]
This is because $j-\delta_{\leq j} = i-\delta_{\leq i}$ by the definition of~$i$,
so that the coefficients of $e_k$ agree for all $k\geq i$;
for $k<i$, however, we also have $k<j$ and thus 
$k-\delta_{\leq k}+\delta_{\leq j}\geq k$,
so that the coefficient in~$t^j\tilde e_i$ of~$e_k$ is zero.
Therefore, $t^j\tilde e_i$ equals zero in~$N/\<x\>$.

However, if some $R$-linear combination
$y = \sum_{i\in I} r_i\tilde e_i\in N$ is zero in $N/\<x\>$,
then $r_i\in\mf m^{j(i)}N$ for each $i$.  Otherwise, write $y=sx$. 
Let $i$ be minimal so that $r_i\not\in\mf m^{j(i)}N$, and let $j=j(i)$.
If $y$ is expanded in terms of the $e_k$, then the least $k$ such that
$e_k$ has a nonzero coefficient is $k=i$.  
Let $i'$ be the greatest element of~$I$ less than~$i$.
Since the coefficient of $e_{i'}$ in~$y$ is zero,
the $\mf m$-valuation of $s$ must be greater than or equal to
$i' - (i' - \delta_{\leq i'}) = \delta_{\leq i'} = \delta_{\leq j}$,
in view of the definition of~$x$.
(Or, if there is no element of~$I$ less than~$i$, then consideration of
the coefficient of $e_i$ in $x$ yields the same conclusion.)
But then the $\mf m$-valuation of the coefficient of $e_i$ in $y$ 
is greater than or equal to
$(i - \delta_{\leq i}) + \delta_{\leq j} = j$,
contradicting our assumption on~$i$.  

It follows that the $R$-module generated by the $\tilde e_i$ is isomorphic to
\[\bigoplus_{i\in I} R/\mf m^{j(i)},\]
wherein $\{j(i) : i\in I\}$ is the set of all indices $j$ for which
$\delta_j=0$ and $\delta_{j+1}=1$.  
The elements~$\tilde e_i$ in fact generate $Q$, 
by a triangularity argument between the $\tilde e_i$ and the~$e_i$.
We conclude that the sequences $d_i(N)-d_i(N')$ and $\delta_i$ are equal.
\end{proof}

\begin{proof}[Proof of Corollary~\ref{cor:DVR M1}]
Let $\nu_\alpha = \dim_{R/\mf m}(\<e_\alpha\>/\<x_\alpha e_\alpha\>)$;
this is the maximum of $\ell_\alpha$ and the $\mf m$-valuation of $x_\alpha$.
Suppose first that $x_\alpha=0$ for all $\alpha$ except for a single list
$A=\{\alpha_1,\ldots,\alpha_{|A|}\}$ such that both
$(\nu_{\alpha_i})$ and $(\ell_{\alpha_i} - \nu_{\alpha_i})$ are strictly increasing sequences.
To avoid preliferation of subscripts we will write $\nu_i\doteq\nu_{\alpha_i}$
and $\ell_i\doteq\ell_{\alpha_i}$.  

The condition~\eqref{eq:DVR d} is vacuous when $x_\alpha=0$.
The sequence $d_\bullet$ that we obtain from \eqref{eq:DVR d} for the $\alpha\in A$ is
\[0^{\nu_1}1^{\ell_1-\nu_1}
0^{\nu_2-\nu_1}1^{\ell_2-\nu_2-\ell_1+\nu_1}
0^{\nu_3-\nu_2}1^{\ell_3-\nu_3-\ell_2+\nu_2}\ldots,\]
exponents indicating repetition.  
For this sequence, if the sufficiency argument of Proposition~\ref{p:DVR M1}
is run with the same choice of generators $\{e_\alpha\}$, the element $x$
produced to generate the kernel is the same one we have chosen here
(up to automorphisms of the cyclic summands $\<e_\alpha\>$).  
So the corollary is proven in this case.

Now suppose two indices $\alpha$ and~$\alpha'$ are such that
\begin{equation}\label{eq:DVR d1}
\nu_\alpha\leq\nu_{\alpha'}
\quad\mbox{and}\quad
\ell_\alpha-\nu_\alpha\geq\ell_{\alpha'}-\nu_{\alpha'}.
\end{equation}
The inequality \eqref{eq:DVR d} holds if and only if the $(\nu_\alpha+1)$\/th 0
of $d_\bullet$, if any, follows at least $\ell_\alpha-\nu_\alpha$ 1s.
Hence, \eqref{eq:DVR d} for~$\alpha'$ is implied by \eqref{eq:DVR d} for~$\alpha$,
and thus $\alpha'$ is irrelevant for computing $d_\bullet$.  
Moreover, inequalities \eqref{eq:DVR d1}
ensure that we may change our basis for $N$ by adding a multiple
of $e_\alpha'$ to $e_\alpha$, 
yielding another generator $\tilde e_\alpha$ of~$\<e_\alpha\>$,
so that 
\[x_\alpha e_\alpha + x_{\alpha'} e_{\alpha'} = \tilde x_\alpha\tilde e_\alpha\]
for some $\tilde x_\alpha$ with the same $\mf m$-valuation as $x_\alpha$.

By repeatedly making such changes of basis, we may, with no changes to
the sequence $d_\bullet$ that will be computed, assume that $x_\alpha=0$
for all $\alpha$ except for a set no two of whose members $\alpha$, $\alpha'$
satisfy \eqref{eq:DVR d1}.  But such a set may be ordered
so that $(\nu_{\alpha_i})$ and $(\ell_{\alpha_i} - \nu_{\alpha_i})$ 
are both strictly increasing, and this reduces to the first case.
\end{proof}

Having control over condition~(M1), we turn to the axiom~(M).

\begin{proposition}\label{p:DVR M2}
Assume that the residue field of $R$ has order greater than~$2$.
Let $M(\emptyset)$, $M(1)$, $M(2)$, and $M(12)$ be f.g.\ $R$-modules.
There exist four surjections with cyclic kernels forming a pushout square
\[\xymatrix{
\ar@{}[dr]|{\mbox{\LARGE$\lrcorner$}}
M(\emptyset)\ar^\phi[r]\ar_\psi[d] & M(1)\ar^{\psi'}[d] \\
M(2)\ar_{\phi'}[r] & M(12)
}\]
if and only if
\begin{enumerate}
\item[(L1)] the source and target of each map satisfy Proposition~\ref{p:DVR M1};
\item[(L2a)]
 for each $n\geq1$, 
\[d_{\leq n}(M(\emptyset)) - d_{\leq n}(M(1)) - d_{\leq n}(M(2)) + d_{\leq n}(M(12)) \geq 0;\]
\item[(L2b)]
for any $n\geq1$ such that $d_n(M(1))\neq d_n(M(2))$, equality holds above:
\[d_{\leq n}(M(\emptyset)) - d_{\leq n}(M(1)) - d_{\leq n}(M(2)) + d_{\leq n}(M(12)) = 0.\]
\end{enumerate}
\end{proposition}

The numbering of these conditions is chosen to agree with the
numbering of the axioms for a quasi-arithmetic matroid in
Corollary~\ref{c:arithmetic}.


Condition (L2a) asserts that $A \mapsto -d_{\leq n}(M(A))$
is a \emph{submodular} function.

\begin{proof} 
\noindent\emph{Necessity.}  Condition~(L1) is clear from the fact that
axiom~(M) implies condition~(M1).  

Tensoring the matroid $M$ with
with $R/\mf m^n$ gives a matroid $M' \doteq M\otimes(R/\mf m^n)$ over that ring.
All of its modules are of finite length.  
Now regard these modules $M'(A)$ as $R/\mf m$-vector spaces.
The maps $M'(A)\to M'(Ab)$ given by (M1) remain surjective,
and the pushout diagrams in (M) remain pushouts, since 
surjectivity and pushout-hood
can be checked set-theoretically.  Accordingly,
$M'$ can be interpreted as a polymatroid over $R/\mf m$, that is,
a classical polymatroid.  
The negative of the corank function of a polymatroid is submodular,
and this is Condition~(L2a).

As for condition~(L2b), 
suppose that the inequality of~(L2a) were strict.
Let the kernel of~$\phi$ be $\<x\>$, and the kernel of~$\psi$ be $\<y\>$,
so that the kernel of the composite $\psi'\circ\phi=\phi'\circ\psi$ is $\<x,y\>$.
So our assumption is
\[\dim \<x\>/(\<x\>\cap\mf m^nN) + \dim \<y\>/(\<y\>\cap\mf m^nN) >
\dim \<x,y\>/(\<x,y\>\cap\mf m^nN)  \]
where all dimensions are over $R/\mf m$.  (Note that the non-strict version
of this inequality manifestly holds, providing another proof of~(L2a).)
That is, there exist $r,s\in R$ such that $sy-rx\in\mf m^nN$,
but neither $rx$ nor $sy$ is in $\mf m^nN$.

Now, suppose that $d_n(M(\emptyset))-d_n(M(1))=1$.  
By the proof of Proposition~\ref{p:DVR M1},
the module $(\<x\>\cap\mf m^{n-1}N)/(\<x\>\cap\mf m^nN)$
is nontrivial, i.e.\ there exists $q\in R$ so that 
\[qx\in\mf m^{n-1}N\setminus\mf m^nN.\]
Because $qx\in\mf m^{n-1}N$ and $rx\not\in\mf m^nN$, we have that $r$ divides
$q$ in~$R$, say $r=pq$.  Then 
\[psy-qx = p(sy-rx)\in\mf m^nN\]
and by adding, we get $psy\in\mf m^{n-1}N\setminus\mf m^nN$,
which implies that $d_n(M(\emptyset))-d_n(M(2))=1$.
Of course the same holds with the roles of $1$ and $2$ in the ground set reversed,
so that $d_n(M(1)) = d_n(M(2))$.  By contradiction, (L2b) is proved.

\noindent\emph{Sufficiency.}
Suppose the modules $M(A)$ satisfy (L1), (L2a), (L2b).  
For each $A\subseteq\{1,2\}$, write 
\[M(A) = \bigoplus_{\alpha=1}^s R/\mf m^{\ell^A(\alpha)},\]
where the $\ell^A(\alpha)\in\mathbb N\cap\{\infty\}$ form a non-increasing sequence.
Let $\overline d_{\leq\ell^A(\alpha)}(\phi)$ denote $\ell^A(\alpha) - d_{\leq\ell^A(\alpha)}(\phi)$.
We will abbreviate these notation by omitting the braces and commas from the superscript when $A$ is a particular set,
and in particular use $\ell$ in place of $\ell^\emptyset$.

Let $e_\alpha$ be a generator for the summand $R/\mf m^{\ell(\alpha)}$ of~$M(\emptyset)$.
Let $t\in R$ be a generator of~$\mf m$ and
define the element $x\in M(\emptyset)$ by 
\begin{equation}\label{eq:DVR M2 x}
x = \sum_\alpha t^{\overline d_{\leq\ell(\alpha)}(\phi)} e_\alpha.
\end{equation}
By assumption (L1), if $f$ is one of the maps in the pushout,
the sequence $d_\bullet(f)$ has elements drawn from $\{0,1\}$,
so by the proof of Proposition~\ref{p:DVR M1}, this arranges that $M(\emptyset)/\<x\>\cong M(1)$.
Also as in that proof, define the elements $\tilde e_\alpha\in M(\emptyset)$ by
\[\tilde e_\alpha = \sum_{\beta=1}^\alpha t^{\overline d_{\leq\ell(\beta)}(\phi) - \overline d_{\leq\ell(\alpha)}(\phi)} e_\beta.\]
Then the images $f_\alpha$ of the $\tilde e_\alpha$ in $M(\emptyset)/\<x\>$ are a minimal set of generators,
except that the last one, $f_s$, may be zero.

If we construct a linear combination of the $f_\alpha$
using any family of coefficients $\tilde y_\alpha\in R$ such that 
$\val\tilde y_\alpha$ is equal to $\overline d_{\leq\ell^1(\alpha)}(\psi')$,
where $\val:R\to\mathbb N$ is the valuation of~$R$,
then the quotient $\big(M(\emptyset)/\<x\>\big)/\<\sum_\alpha y_\alpha f_\alpha\>$ by this linear combination
will be isomorphic to $M(12)$.
So it is enough for us to choose the $\tilde y_\alpha$ to arrange that
$y=\sum_\alpha \tilde y_\alpha\tilde e_\alpha$ satisfies $M(\emptyset)/\<y\>\cong M(2)$.

Expanding, we have
\[
y = \sum_\alpha \tilde y_\alpha\tilde e_\alpha
= \sum_\beta\left(\sum_{\alpha\geq\beta} t^{\overline d_{\leq\ell(\beta)}(\phi) - \overline d_{\leq\ell(\alpha)}(\phi)}\,\tilde y_\alpha\right) e_\beta.\]
Assign the name $y_\beta$ to the inner sum, so that $y = \sum_\beta y_\beta e_\beta$.  Then
\begin{align}
y_\beta &= \sum_{\alpha\geq\beta} t^{\overline d_{\leq\ell(\beta)}(\phi) - \overline d_{\leq\ell(\alpha)}(\phi)}\,\tilde y_\alpha \notag
\\&= \tilde y_\beta + t^{\overline d_{\leq\ell(\beta+1)}(\phi) - \overline d_{\leq\ell(\beta)}(\phi)}
\sum_{\alpha\geq\beta+1} t^{\overline d_{\leq\ell(\beta+1)}(\phi) - \overline d_{\leq\ell(\alpha)}(\phi)}\,\tilde y_\alpha \notag
\\&= \tilde y_\beta + t^{\overline d_{\leq\ell(\beta+1)}(\phi) - \overline d_{\leq\ell(\beta)}(\phi)} y_{\beta+1}. \label{eq:5.4suff1}
\end{align}
Let us momentarily name the latter summand $z_\beta := t^{\overline d_{\leq\ell(\beta+1)}(\phi) - \overline d_{\leq\ell(\beta)}(\phi)} y_{\beta+1}$.
Again by Proposition~\ref{p:DVR M1}, to achieve $M(\emptyset)/\<y\>\cong M(2)$
it is enough to arrange that $\val y_\beta = \overline d_{\leq\ell(\beta)}(\psi)$.
By choosing the $\tilde y_\beta$ one at a time in order of decreasing index,
we can consider the instance of the above equation for each value of $\beta$ separately.
A choice of $\tilde y_\beta$ will be possible just if 
the desired value of the triple
$(\val y_\beta, \val\tilde y_\beta, \val z_\beta)$ lies in the set
\[\{(\val(a+b),\val a,\val b) : a,b\in R\}\]
because, given such elements $a$ and $b$, we get $z_\beta = kb$ for some $k\in R\setminus\mf m$,
and then we make the choice $\tilde y_\beta = ka$.
It is easy to see that the above set is
\[\{(v_1,v_2,v_3)\in\mathbb N^3 : \mbox{the minimum of $\{v_1,v_2,v_3\}$ occurs at least twice}\}.\]
The inclusion left to right is the additive axiom for nonarchimedean valuations.
Right to left, the key fact is that the sum of two elements of $R\setminus\mf m$
can have any valuation; this uses $|R/\mf m|>2$ to realise the tuple $(0,0,0)$
(and more generally $(v,v,v)$).

As just stated, the desired valuation of the left hand side of~\eqref{eq:5.4suff1} 
is $\val y_\beta = \overline d_{\leq\ell(\beta)}(\psi)$,
while those of the two summands on the right hand side are respectively
$\val\tilde y_\beta = \overline d_{\leq\ell^1(\beta)}(\psi')$
and
$\val z_\beta = \overline d_{\leq\ell(\beta+1)}(\phi) - \overline d_{\leq\ell(\beta)}(\phi) + \overline d_{\leq\ell(\beta+1)}(\psi)$.
So what we must show is that 
\begin{equation}\label{eq:DVR M2 1}
\overline d_{\leq\ell^1(\beta)}(\psi') \geq
\min\{\overline d_{\leq\ell(\beta)}(\psi),\, 
\overline d_{\leq\ell(\beta+1)}(\phi) - \overline d_{\leq\ell(\beta)}(\phi) + \overline d_{\leq\ell(\beta+1)}(\psi)\}
\end{equation}
and that equality holds if the two terms of the minimum are different.

To prove the inequality, we consider two cases according to how $\ell^1(\beta)$ compares to $\ell^2(\beta)$. 
If $\ell^2(\beta)\leq\ell^1(\beta)$, then $d_n(\psi)=1$ for all $\ell^1(\beta)<n\leq\ell(\beta)$,
which implies that
$\overline d_{\leq\ell(\beta)}(\psi) = \ell(\beta) - d_{\leq\ell(\beta)}(\psi)$ is equal to
$\ell^1(\beta) - d_{\leq\ell^1(\beta)}(\psi) = \overline d_{\leq\ell^1(\beta)}(\psi)$.
Assumption (L2a) says that
\[d_{\leq n}(\psi) = d_{\leq n}(M(\emptyset))-d_{\leq n}(M(2)) \geq 
d_{\leq n}(M(1))-d_{\leq n}(M(12)) = d_{\leq n}(\psi'),\]
or equivalently
\begin{equation}\label{eq:DVR M2 2}
\overline d_{\leq n}(\psi) \leq \overline d_{\leq n}(\psi'),
\end{equation}
and at $n=\ell^1(\beta)$ this now proves that 
the first term in the minimum in~\eqref{eq:DVR M2 1} is a lower bound for $\overline d_{\leq\ell^1(\beta)}(\psi')$.
The other case is that $\ell^2(\beta)\geq\ell^1(\beta)$.
In this case, $\overline d_{\leq\ell^1(\beta)}(\psi)$ is equal to
\[\ell(\beta+1) - \ell^1(\beta) + \overline d_{\leq\ell(\beta+1)}(\psi) = 
\overline d_{\leq\ell(\beta+1)}(\phi) - \overline d_{\leq\ell^1(\beta)}(\phi)+\overline d_{\leq\ell(\beta+1)}(\psi),\]
in which the first two terms account for the indices $\ell(\beta+1)<n\leq\ell^1(\beta)$ for which $d_n(\psi)=0$,
which is true because $\ell^1(\beta)\leq\ell^2(\beta)$.
We have as well $\overline d_{\leq\ell^1(\beta)}(\phi) = \overline d_{\leq\ell(\beta)}(\phi)$, 
since $d_n(\phi)=1$ in this range.  
Therefore, again using~\eqref{eq:DVR M2 2} at $n=\ell^1(\beta)$ we prove that 
the second term on the right of~\eqref{eq:DVR M2 1} is a lower bound for the left.

We have included $\ell^1(\beta)=\ell^2(\beta)$ in both of the above cases,
so in this event the arguments above show that both of the terms of the minimum in~\eqref{eq:DVR M2 1}
are equal to~$\overline d_{\leq\ell^1(\beta)}(\psi)$, and thus they equal each other.
So by contraposition, if the two terms of the minimum are not equal,
then $\ell^1(\beta)\neq\ell^2(\beta)$.
This implies that $d_n(M(1))\neq d_n(M(2))$ 
either for $n=\ell^1(\beta)$ (if $\ell^1(\beta)>\ell^2(\beta)$)
or for $n=\ell^1(\beta)+1$ (if $\ell^1(\beta)<\ell^2(\beta)$).
In either case we can invoke (L2b) in place of the respective invocation of (L2a) in the preceding paragraph,
proving that the inequality we showed there is in fact an equality, as needed.
\end{proof}

When the residue field is of order~2, some further possibilities for the data $d_\bullet(M(A))$ are ruled out.

\begin{propchar2}\label{p:DVR M2 char 2}
Assume that the residue field of $R$ has order~$2$.
Let $M(\emptyset)$, $M(1)$, $M(2)$, and $M(12)$ be f.g.\ $R$-modules.
There exist four surjections with cyclic kernels forming a pushout square
\[\xymatrix{
\ar@{}[dr]|{\mbox{\LARGE$\lrcorner$}}
M(\emptyset)\ar^\phi[r]\ar_\psi[d] & M(1)\ar^{\psi'}[d] \\
M(2)\ar_{\phi'}[r] & M(12)
}\]
if and only if conditions (L1), (L2a) and (L2b) of Proposition~\ref{p:DVR M2} hold, together with
\begin{enumerate}
\item[(L2c)]
if
\[d_{\leq n}(M(\emptyset)) - d_{\leq n}(M(1)) - d_{\leq n}(M(2)) + d_{\leq n}(M(12)) = 0\]
holds for some $n=n_0\geq1$, then either it also holds for at least one of $n=n_0-1$ and $n=n_0+1$,
or $M(1)$ has multiple summands isomorphic to $R/\mf m^n$.
\end{enumerate}
\end{propchar2}

\begin{proof}
Given the four modules, conditions (L1), (L2a), (L2b)
can still be proved using the proof of~Proposition~\ref{p:DVR M2}.
Therefore we may assume these conditions attain, and use (L1) to construct the elements
in the sufficiency half of that proof.
Up to isomorphism of $M(\emptyset)$, the only choice of an element $x\in M(\emptyset)$
such that $M(\emptyset)/\<x\>\cong M(1)$ is the one defined above in~\eqref{eq:DVR M2 x},
so we may fix this choice of~$x$ and follow that analysis to ascertain when a choice of~$y$ exists.
Conversely, to adapt the proof of~Proposition~\ref{p:DVR M2} to prove
that our stronger conditions are sufficient for the existence of our morphisms,
what we must show is that they avoid the potential problem pointed out above
where the assumption on $|R/\mathfrak m|$ was invoked.

If $|R/\mf m|=2$, we have
\begin{multline*}
\{(\val(a+b),\val a,\val b) : a,b\in R\} \\=
\{(v_1,v_2,v_3)\in\mathbb N^3 : \mbox{the minimum of $\{v_1,v_2,v_3\}$ occurs \emph{exactly} twice}\}.
\end{multline*}
Another difference will be that here, unlike in the proof of~Proposition~\ref{p:DVR M2},
it will not be enough to assume that $\val\tilde y_\alpha$ equals $\overline d_{\leq\ell^1(\alpha)}(\psi')$,
or that $\val\tilde y_\alpha$ equals $\overline d_{\leq\ell(\alpha)}(\psi)$, for all $\alpha$.
Recall that these assumptions were stronger than the necessary and sufficient condition in Corollary~\ref{cor:DVR M1}.

So what our task reduces to is proving that, supposing (L1), (L2a), (L2b) hold,
(L2c) is equivalent to the assertion that 
sequences of elements $v(\beta),v'(\beta)\in\mathbb N\cup\{\infty\}$ can be chosen
to meet the requirements of Corollary~\ref{cor:DVR M1} and the corresponding strengthening of equation~\eqref{eq:DVR M2 1}.
To spell these conditions out, the former is that 
$v(\beta)\geq \overline d_{\leq\ell(\beta)}(\psi)$ 
(respectively, $v'(\beta)\geq \overline d_{\leq\ell^1(\beta)}(\psi')$) 
for each $\beta$,
and that equality holds for at least one of the values $\beta = \alpha, \alpha+1, \ldots, \alpha'$
whenever this is a maximal consecutive sequence of $\beta$ for which the $d_{\leq\ell(\beta)}(\psi)$ are equal,
except that this condition is not imposed when 
$\ell(\alpha)=\cdots=\ell(\alpha')$ and $d_{\ell(\alpha)}(\psi) = 0$
(respectively, the same is imposed for $\ell^1$ and $\psi'$).
The latter is that
\begin{equation}\label{eq:DVR M2 char 2 1}
v'(\beta) \geq
\min\{v(\beta), \, 
\overline d_{\leq\ell(\beta+1)}(\phi) - \overline d_{\leq\ell(\beta)}(\phi) + v(\beta+1)\}
\end{equation}
for each $\beta$, with equality if and only if the two terms of the minimum are not equal.

Suppose that (L2c) is violated at~$n$.  
This implies that $d_n(\psi)=0$ and $d_n(\psi')=1$ while $d_{n+1}(\psi)=1$ and $d_{n+1}(\psi')=0$.
It follows that $n=\ell^1(\beta)$ for some $\beta$.  Let $\alpha$ and $\alpha'$ be as in the last paragraph.
Since $d_{n+1}(\psi') = d_{n+1}(\phi') = 0$
we have $\ell^1(\beta) = \ell^2(\beta) = n$ for all $\alpha\leq\beta\leq\alpha'$,
so by the last clause of~(L2c) we have $\alpha=\alpha'$
(because $d_n(\psi') = d_n(\phi') = 1$ we are not in the case where this last clause is inapplicable).
Corollary~\ref{cor:DVR M1} now demands that $v(\alpha)$ equals $\overline d_{\leq\ell^1(\alpha)}(\psi')$.
This implies
\[\overline d_{\leq\ell^1(\alpha)}(\psi')=v'(\alpha)\geq v(\alpha)
\geq\overline d_{\leq\ell(\alpha)}(\psi) = \overline d_{\leq\ell^1(\alpha)}(\psi) 
=\overline d_{\leq\ell^1(\alpha)}(\psi'),\]
in which the first inequality is from \eqref{eq:DVR M2 char 2 1}, 
and the last equality is the modularity assumption from~(L2c).
We thus get $v'(\alpha)=v(\alpha)$.
A parallel chain of inequalities shows that 
$v'(\alpha)=\overline d_{\leq\ell(\alpha+1)}(\phi) - \overline d_{\leq\ell(\alpha)}(\phi) + v(\alpha+1)$.
But this now violates the proviso on equality after \eqref{eq:DVR M2 char 2 1}.

If (L2c) is not violated for any~$n$, we will construct the sequences $v(\beta)$ and $v'(\beta)$.
Our choice will be $v(\beta)=\overline d_{\leq\ell(\beta)}(\psi)$ 
and $v'(\beta)=\overline d_{\leq\ell^1(\beta)}(\psi')$
except in certain situations where this choice would violate the proviso after~\eqref{eq:DVR M2 char 2 1}.
To wit, the exceptions are the following four situations
where $\ell^1(\beta)=\ell^2(\beta)=:n$ and the modularity equality holds for the quantities $d_{\leq n}$.
\begin{enumerate}
\item If $\ell^1(\beta)=\ell^2(\beta)=n$ for all $\beta$ in a range $\alpha\leq\beta\leq\alpha'$
and $\alpha<\alpha'$ strictly, then set $v(\alpha)=\infty$ and
$v'(\alpha+2)=\cdots=v'(\alpha')=\infty$.\\
If this does not attain, then:
\item if $d_n(\psi)=d_{n+1}(\psi)=0$, set $v(\beta)=\infty$;
\item if $d_n(\psi)=d_{n+1}(\psi)=1$, set $v(\beta+1)=\infty$;
\item if $d_n(\psi')=d_{n+1}(\psi')$, set $v'(\beta)=\infty$.
\end{enumerate}
This covers every case allowed for by~(L2c), 
so what remains to be checked is that the requirements of Corollary~\ref{cor:DVR M1} are met.
Item (1) is not a problem for Corollary~\ref{cor:DVR M1} because
$v(\alpha+1)=\overline d_{\leq\ell(\alpha+1)}(\psi)$ 
and $v'(\alpha)=\overline d_{\leq\ell^1(\alpha)}(\psi)$ 
have not been altered.
Items (2) through~(4) meet the requirements because there exists some $\gamma$ 
in the range $\alpha\leq\gamma\leq\alpha'$ discussed above such that
$d_{\ell(\gamma)}(\psi)=0$ and $d_{\ell(\gamma)+1}(\psi)=1$ (or the same for $\psi'$),
and $v(\gamma)$ (resp.\ $v'(\gamma)$) has not been altered.
\end{proof}

\begin{exchar2}
Let $R = \mathbb Z_2$, the localisation of the integers at~$2$.
The simplest example of four modules which do not form an $R$-matroid
despite satisfying the conditions of Proposition~\ref{p:DVR M2} are 
$M(\emptyset) = R\oplus R/\<2\>$,
$M(1) = R/\<4\>$,
$M(2) = R/\<4\>$, and
$M(12) = R/\<2\>$.
These violate condition~(L2c) because the unique positive integer for which the modularity equality
\[d_{\leq n}(M(\emptyset)) - d_{\leq n}(M(1)) - d_{\leq n}(M(2)) + d_{\leq n}(M(12)) = 0\]
holds is $n=2$, and $M(1)$ has only a single summand $R/\< 2^2\>$.  
The reader is invited to check that two suitable elements $x,y\in M(\emptyset)$ cannot be constructed.
On the other hand, taking the direct sum of each of these modules with $R/\< 4\>$
does yield an $R$-matroid:
putting $N = R\oplus R/\< 4\>\oplus R/\<2\>$
and taking $x=(2,0,1), y=(2,2,1)$ in~$N$ gives
$N/\<x\>\cong N/\<y\>\cong (R/\<4\>)^2$ and
$N/\<x,y\>\cong R/\<4\>\oplus R/\<2\>$.

In the same way,
$M(\emptyset) = \mathbb Z\oplus \mathbb Z/\<2\>$,
$M(1) = \mathbb Z/\<4\>$,
$M(2) = \mathbb Z/\<4\>$, and
$M(12) = \mathbb Z/\<2\>$
give a non-example of a matroid over $\mathbb Z$ on two elements,
which becomes a matroid on direct sum with $\mathbb Z/\<4\>$.
\end{exchar2}

By the time we come to three-element matroids over~$R$, 
there are already nontrivial conditions on the
functions $d_{\leq n}$ beyond their negatives being submodular.

\begin{proposition}\label{p:DVR 3}
Let $M$ be a matroid over~$R$ on the ground set~$[3]$,
and let $n$ be a natural or $\infty$.  Then, among the three quantities
\[
d_{\leq n}(M(1)) + d_{\leq n}(M(23)),
d_{\leq n}(M(2)) + d_{\leq n}(M(13)),
d_{\leq n}(M(3)) + d_{\leq n}(M(12)),
\]
the minimum is achieved at least twice.
\end{proposition}

\begin{proof}
If $M'$ is a two-element matroid over~$R$, let $s_{\leq n}(M)$ denote the alternating
sum appearing in conditions~(L2a,b).  
The matroid $M$ has 6 minors with two elements.  By adding 
\[ d_{\leq n}(M(\emptyset)) - d_{\leq n}(M(1)) - d_{\leq n}(M(2)) - d_{\leq n}(M(3)) \]
to the three quantities in the proposition, we obtain the three values 
$s_{\leq n}(M\setminus a)$ for the deletions; by adding instead
\[ d_{\leq n}(M(123)) - d_{\leq n}(M(12)) - d_{\leq n}(M(13)) - d_{\leq n}(M(23)), \]
we recover the three values $s_{\leq n}(M/a)$ for the contractions.
So it is equivalent to prove that either of these sets of three attains its minimum 
multiple times.

We use induction on~$n$.  As base case we take $n=0$, and have that 
$s_{\leq 0}(M')=0$ for any $M'$.  So let $n>0$.  
Suppose first, as $A$ varies over subsets of $[3]$, 
that $d_n(M(A))$ depends only on $|A|$.  
In this case, the three sums of form
\[d_n(M(\emptyset))-d_n(M(a))-d_n(M(b))+d_n(M(ab))\]
are equal (as of course are the three sums of form
\[d_n(M(c))-d_n(M(ac))-d_n(M(bc))+d_n(M(abc)). \quad)\]
Therefore the differences $s_{\leq n}(M\setminus a) - s_{\leq n-1}(M\setminus a)$
are all equal, and the induction step succeeds. 

So suppose this is not the case, and there are two sets $A$ and $B$
with $|A|=|B|\in\{1,2\}$, for which $d_n(M(A))\neq d_n(M(B))$.  
We will proceed assuming that $|A|=|B|=1$; the argument for the
other case is exactly analogous 
(in fact, the two cases are exchanged by replacing $M$ by its dual).  
Since there are only two possible values for $d_n(M(A))$ with $|A|=1$, 
namely $d_n(M(\emptyset))$ and $d_n(M(\emptyset))-1$, 
two of the $d_n(M(A))$ with $|A|=1$ are equal and are unequal to the third.
Without loss of generality suppose $d_n(M(1)) = d_n(M(2))\neq d_n(M(3))$.
By condition~(L2b), 
it follows that 
$s_{\leq n}(M\setminus 1) = s_{\leq n}(M\setminus 2) = 0$.
Since $s_{\leq n}(M\setminus 3)$ is nonnegative by condition~(L2a), 
this completes the induction for finite $n$.  

Finally, the case $n=\infty$ holds because if any $d_{\leq\infty}(A)$ is finite, 
then $d_{\leq n}(A)$ must be eventually constant and equal to $d_{\leq\infty}(A)$.
If the minimum in the proposition is not $\infty$, there is nothing to prove;
if this minimum is finite, the claim follows on replacing $n$ by
a sufficiently large finite number.
\end{proof}

Suppose given a matroid $M$ over~$R$
with ground set $E$.
For $A\subseteq E$, define $p_A$ to be $d_{\leq n}(M(A))$.  
Applying Proposition~\ref{p:DVR 3} to all three-element minors of $M$: 
the result can be restated to say that the \emph{tropicalizations}
of the relations
\begin{equation}\label{eq:Pluecker 3}
p_{Ab}p_{Acd} - p_{Ac}p_{Abd} + p_{Ad}p_{Abc} = 0
\end{equation}
hold of the numbers $p_\bullet$, where we continue abbreviating
$A\cup\{b,c\}$ as $Abc$ and similarly.  

For background on tropical geometry, see \cite{Maclagan-Sturmfels}.
We say a bare minimum here: tropicalization 
is a procedure transforming algebraic varieties
to tropical varieties, combinatorial ``shadows'' thereof, which are the sets of
points on which the tropicalizations of all elements of their ideal of defining equations vanish.
In our situation without a valued field,
the tropicalization of a polynomial $f=\sum_{a\in A} c_a x^a$ in variables $x_1,\ldots,x_d$
is said to vanish at those points $(x_i)$ where,
of linear forms $\sum_i a_ix_i$ corresponding to the monomials in $f$,
the minimum value is attained by two or more of the forms.

The relations \eqref{eq:Pluecker 3} are among the Pl\"ucker
relations for the full flag variety (of type $A$).
A Pl\"ucker relation is a quadratic relation among the $p_\bullet$ 
arising from the straightening algorithm for Young tableaux.
The full flag variety 
has a tropical analogue, the \emph{tropical flag Dressian}~\cite{Haque} cut out by the
\emph{tropical Pl\"ucker relations}, which arise from tropicalizing
those Pl\"ucker relations with the fewest terms.
To be precise, the tropical Pl\"ucker relations arise from the Pl\"ucker relations of the form 
\begin{equation}\label{eq:Pluecker gen}
\sum_{i\in T\setminus S}\pm\, p_{S\cup\{i\}}\, p_{T\setminus\{i\}} = 0
\end{equation}
where $S$ and~$T$ are subsets of~$E$ satisfying $|S|+1\leq|T|-1$.
When one considers only those $p_A$ with $|A|=r$
one can restrict to the relations with $|S|=r-1$ and $|T|=r+1$.
These relations define the \emph{Dressian}
$\mathrm{Dr}(r,n)$, which is one Grassmannian-like space in tropical geometry.
It is the parameter space for tropical linear spaces \cite{HJS}.   
That is, there is a tropical linear space determined by~$(p_A : |A|=r)$
if and only if this point lies on the Dressian.
Observe that if $|S|+1=|T|-1$ and $S\subseteq T$, the relation~\ref{eq:Pluecker gen} 
vacuously says $0=0$, so the smallest
tropical Pl\"ucker relations have three terms.

Considering all the tropical Pl\"ucker relations yields the flag Dressian, 
which is to the full flag variety as the Dressian is to the Grassmannian:
while points in the Dressian are tropical linear spaces,
points in the flag Dressian are full flags of tropical linear spaces,
where a flag is defined by satisfying incidence conditions.


\begin{proposition}\label{p:DVR exchange}
Define $p_A = d_{\leq n}(M(A))$,
where $M$ is a matroid over~$R$, and $n$ is a natural or $\infty$.
Then the collection of $p_A$ for all subsets $A\subseteq E$
satisfies every tropical Pl\"ucker relation.
\end{proposition}

\begin{corollary}\label{conj:DVR exchange}
The collection of $p_A$ in Proposition~\ref{p:DVR exchange}
gives a point on the flag Dressian.
\end{corollary}

In particular, for every $0<r<n$,
the point $(p_A : |A|=r)$ lies on the Dressian $\mathrm{Dr}(r,n)$.
As another equivalent formulation,
if $M$ is a matroid over $R$, then in
the regular subdivision of the hypersimplex $\conv\{\sum_{i\in A} e_i\}$ wherein
the height of vertex $A$ is $p_A = d_{\leq n}(M(A))$, 
then all maximal faces of this subdivision are matroid polytopes. 
For more on these correspondences see \cite{Speyer} and 
\cite[Section 4.4]{Maclagan-Sturmfels}.

\begin{remark}\label{refereequestion}
%
In general, $M$ is not determined by
the collection $(p_A : A\subseteq E)$,
for any fixed $n$.
For instance, they do not distinguish the two one-element matroids
$M(\emptyset)=R/(\mf m^2)$, $M(1)=R/\mf m$ and
$M(\emptyset)=R/\mf m\oplus R/\mf m$, $M(1)=R/\mf m$.
However, $M$ is determined by the whole family of tuples 
$(p_A : A\subseteq E)$ as $n$ varies
(geometrically, by a map from a tropical ray into each Dressian).  
\end{remark}

\begin{corollary}\label{cor:valuated matroid}
Let $M$ be a matroid over a DVR $(R,\mf m)$.  Then the function
$A\mapsto \dim_{R/\mf m} M(A)$ makes the generic matroid of~$M$ into a
valuated matroid, in the sense of Dress and Wenzel~\cite{DW}.
\end{corollary}

To be precise, our sign convention is the opposite of the one adopted in~\cite{DW};
for perfect agreement we would have to negate this function.  But our sign convention
is frequently adopted in tropical geometry, see e.g.\ \cite{Maclagan-Sturmfels}.

\begin{proof}
Choose $n\gg0$ sufficiently larger than the greatest length of any 
finite length summand of a module $M(A)$.  The lengths of 
$M(A)\otimes R/\mf m^n$ for $A$ not a spanning set of the generic matroid
are sufficiently greater than these lengths when $A$ is a spanning set,  
since $A$ is a spanning set of the generic matroid if and only if
$M(A)$ has $R$ as a summand.  

The axiom of Dress and Wenzel for the valuation $v$ of a valuated matroid is that, 
given bases $A$ and $B$ and $a\in A\setminus B$, there exists $b\in B\setminus A$
so that $A\setminus\{a\}\cup\{b\}$ and $B\cup\{a\}\setminus\{b\}$ are bases,
and such that
\[v(A)+v(B)\geq v(A\setminus\{a\}\cup\{b\})+v(B\cup\{a\}\setminus\{b\}).\]
The fact that the sets on the right hand side are bases
follows from our choice of~$n$, and the inequality is immediate from the minimum in the 
Pl\"ucker relation 
\[
\sum_{b\in (B\setminus A)\cup\{a\}}\pm\, p_{A\setminus\{a\}\cup\{b\}}\, p_{B\cup\{a\}\setminus\{b\}} = 0
\]
being attained multiply, since $p_Ap_B$ is one of its terms.
\end{proof}

\begin{remark}\label{rem:pruefer}
We expect that matroids over the ring of integers in the Puiseux series,
$R=\bigcup_{n\geq1} \K[\![t^{1/n}]\!]$,
should directly produce tropical objects with coordinates in $\mathbb Q$,
when the length of $R/(t^a)$ is taken as~$a$ for $a\in\mathbb Q$;
and that it is possible to use other valued rings similarly.
Everywhere we have assumed $R$ is a Dedekind domain, we expect it is sufficient 
to let $R$ be a \emph{Pr\"ufer domain}, that is, a ring all of whose localizations at primes
are valuation rings, but not necessarily discrete (which is to say Noetherian).  
Verifying this, and extending those parts of the theory which have relied on Noetherianity, 
is left for future work.
\end{remark}

\begin{proof}[Proof of Proposition~\ref{p:DVR exchange}]
In any Pl\"ucker relation with $|A_{\rm e}|=1$, the constraint $|A_{\rm e}|+|B_{\rm e}| = |B\setminus A|+1$
implies that $B_{\rm e}$ is all of $B\setminus A$.  So, once $A$ and $B$ are chosen,
there is just one exchange relation for each one-element subset (i.e.\ element) of $A\setminus B$.

We will proceed by induction on $|A\setminus B| + |B\setminus A|$.
If $A$ is a subset of $B$, or if $|A\setminus B| = |B\setminus A| = 1$,
there is no nontrivial Pl\"ucker relation.  Thus the first nontrivial case
is $|A\setminus B| = 1$ and $|B\setminus A| = 2$, and this 
is equation~\eqref{eq:Pluecker 3}, which we have established as a base case.

The second nontrivial case is  $|A\setminus B| = |B\setminus A| = 2$,
and we again handle this case separately.  In this case, let
$F=A\cap B$, which equals $B_{\rm f}$ and is one element short of~$A_{\rm f}$.
Suppose without loss of generality that $(A\setminus B)\cup(B\setminus A)=\{1,2,3,4\}$.
Then the tropicalized Pl\"ucker relation to be proved involves the three terms
\[
p_{F12}+p_{F34}, \quad 
p_{F13}+p_{24}, \quad
p_{F14}+p_{F23}.
\]

Consider the six sums $p_{F\cup S_1}+p_{F\cup S_2}+p_{F\cup S_3}$,
where $S_1$, $S_2$, and~$S_3$ are subsets of~$\{1,2,3,4\}$ of
respective sizes 2, 2, and~1, whose union is $\{1,2,3,4\}$, and such that
1 is the unique element appearing twice.  There are six of these sums
(not twelve, because the sum is the same even if $S_1$ and~$S_2$ are exchanged).
Among them there are three sums in which $p_{F12}$ appears.
The remaining two summands in these sums are the tropicalized terms of
a Pl\"ucker relation~\eqref{eq:Pluecker 3}, so their minimum is attained twice.
The same goes for the sums in which $p_{F13}$ appears, or $p_{F14}$.

From there, it follows that if the minimum value of all six of these sums was not attained by,
say, $(S_1,S_2,S_3) = (13,14,2)$, it would be attained at both 
$(13,24,1)$ and $(23,14,1)$, and then by subtracting the common $p_{F\cup\{1\}}$
we would be finished.  Accordingly, and by symmetry permuting $\{2,3,4\}$, 
we may assume that $(13,14,2)$, $(12,14,3)$, and $(12,13,4)$ all attain the minimum.
In particular, they are all equal, and we rearrange to
\[p_{F12}-p_{F2} = p_{F13}-p_{F3} = p_{F14}-p_{F4}.\]

The same argument can be repeated with any of the elements of $\{2,3,4\}$ taking the place of~1.
So if none of those gives the relation sought, we may conclude
$p_{Fij}-p_{Fj} = p_{Fik}-p_{Fk}$
for every $i,j,k$ in $\{1,2,3,4\}$.  But then
\[(p_{F12}-p_{F2})+(p_{F34}-p_{F4}) = (p_{F14}-p_{F4})+(p_{F23}-p_{F2})\]
so that $p_{F12}+p_{F34} = p_{F14}+p_{F23}$, and by symmetry
$p_{F13}p_{F24}$ is equal to both of these as well.  This finishes the case
$|A\setminus B| = |B\setminus A| = 2$.

We finally proceed to the remaining cases, where $|B\setminus A|>2$.
For convenience, write $C = A_{\rm e}\cup B_{\rm e}$.
It is reasonable to do this because, when $|A_{\rm e}|=1$,
there is only one distinct Pl\"ucker relation for $C$:
different partitions of it back into $A_{\rm e}$ and $B_{\rm e}$ of the correct sizes
yield the same relation.

Let $c_1\neq c_2$ be elements of~$C$.  
Let $D$ be a single-element subset of $A_{\rm f}\setminus B$, 
if that set is nonempty, and let $D=\emptyset$ otherwise.
Consider the set $P$ of triples of sets $S=(S_1,S_2,S_3)$
where $S_1$, $S_2$, and $S_3$ are subsets of~$C$ with 
$|S_1|=1$, $|S_2| = |C|-2$, $|S_3| = |C|-1$, 
and such that the multiset union of $S_1$, $S_2$ and~$S_3$ contains 
$c_1$ and $c_2$ with multiplicity~1, and each element of $C\setminus\{c_1,c_2\}$
with multiplicity~2.  
To each $S\in P$ associate the sum
\[\sigma(S) \doteq p_{A_{\rm f}\cup S_1} + p_{B_{\rm f}\cup D\cup S_2} + p_{B_{\rm f}\cup S_3}.\]

A triple $S\in P$ is determined by two elements of~$C$,
namely the unique element $a$ of~$S_1$ and the unique element $b$
of~$C\setminus S_3$.  We write $S_{a,b}$ for this triple, and observe that
$P$ contains exactly those $S_{a,b}$ with either $a=b$ or $a\not\in\{c_1,c_2\}$, $b\in\{c_1,c_2\}$.

In particular, $P$ contains three elements $S_{a,c}$
for each $a\in C\setminus\{c_1,c_2\}$.
The sums $\sigma(S_{a,c})$ for these elements 
are the constant $p_{A_{\rm f}\cup \{a\}}$ plus the tropicalizations of the 
three terms of an instance of~\eqref{eq:Pluecker 3} if $D$ is empty,
respectively the three terms of a Pl\"ucker relation 
where the sets corresponding to $A\setminus B$ and $B\setminus A$
each have size 2 if $D$ is a singleton.  
Accordingly, the minimum value of $\sigma(S_{a,c})$ as $c$ varies is attained twice.

Similarly, $P$ contains $|C|-1$ elements $S_{c,b}$ for each $b\in\{c_1,c_2\}$.
The sums $\sigma(S_{c,b})$ for these elements are
the constant $p_{B_f\cup C\setminus\{b\}}$ added to
tropicalizations of the terms in another Pl\"ucker relation, 
where $c_1$ has been removed from whichever of $A$ and~$B$ it was in,
and $D\subseteq A$ has been added to $B$.
This Pl\"ucker relation is one of those covered by the inductive hypothesis,
since we've shrunk the symmetric difference of $A$ and~$B$.
So the minimum value of $\sigma(S_{c,b})$ as $c$ varies is also attained twice.


Once more, for the $|C|$ elements of $P$ of the form $S_{c,c}$,
each sum $\sigma(S_{c,c})$ is
the constant $p_{B_{\rm f}\cup C\setminus\{c_1,c_2\}}$ plus
the tropicalization of a term in the Pl\"ucker relation
whose tropical vanishing we are concerned with.  So our objective
is to show that the minimum value of $\sigma(S_{c,c})$ is attained twice.

Now let $c_1$ and $c_2$ be chosen so that the number of pairs $a\neq b$
for which $\sigma(S_{a,b})$ attains the minimum value $x=\min_{S\in P}\sigma(S)$
is as small as possible.  We will prove that 
the minimum value of $\sigma(S_{c,c})$ is indeed attained twice.
Suppose not.  We then claim that there exists $a\not\in\{c_1,c_2\}$
such that $\sigma(S_{a,c_1})=\sigma(S_{a,c_2})=x$.  
If this were false, choose $a\neq b$ so that $\sigma(S_{a,b})=x$
(this must be possible, because if $\sigma(S_{a,b})=x$ only
when $a=b$ then the minimum of either $S_{a,c}$ or $S_{c,b}$ as $c$ varies,
whichever is appropriate, is attained just once.)
By the structure of~$P$, $b$ must be $c_1$ or~$c_2$; wlog let it be $c_1$.
That is, we are assuming $\sigma(S_{a,c_1})=x$.
Then by assumption $\sigma(S_{a,c_2})>x$, so
by the three-term Pl\"ucker relation, $\sigma(S_{a,a})=x$.  
Moreover there must exist $a'\neq a$ such that $\sigma(S_{a',c_1})=x$,
by the other Pl\"ucker relation.  If $a'=c_1$ then we have a contradiction
with our first assumption (that the minimum is attained twice); 
otherwise we repeat for $a'$ the argument
we made for~$a$ and have a contradiction with our second.

Thus, we have $\sigma(S_{a,c_1})=\sigma(S_{a,c_2})=x$.  Now, by assumption,
at least one $c\in\{c_1,c_2\}$ has $\sigma(S_{c,c})>x$; without loss of
generality let it be $c_1$.  
Let $P'$ be defined like $P$ except using $\{a,c_2\}$ where $P$ uses $\{c_1,c_2\}$.
For each $S=(S_1,S_2,S_3)$ in~$P$ such that $a\in S_2$,
there is a corresponding $S'=(S_1,S_2\setminus\{a\}\cup\{c_1\},S_3)$ 
in $P'$, with
\[\sigma(S)-\sigma(S')=p_{B_{\rm f}\cup D\cup S_2}-p_{B_{\rm f}\cup D\cup S_2\setminus\{a\}\cup\{c_1\}}.\]
Therefore, if $S$ attains the minimum value $x$ of~$\sigma(S)$ over~$P$,
then $S'$ attains the minimum value of $\sigma(S')$ over~$P'$
(unless this new minimum value is strictly less than 
$x-p_{B_{\rm f}\cup D\cup S_2}-p_{B_{\rm f}\cup D\cup S_2\setminus\{a\}\cup\{c_1\}}$,
in which case only one $S_{a,b}\in P'$ attains it, which contradicts our choice of $c_1$ and~$c_2$).
But in $\{S_{a,b}\in P : a\neq b\}$ we have two elements 
$S_{a,c_1}$ and~$S_{a,c_2}$ without counterparts in~$P'$, 
both of which attain the minimum, whereas in $\{S_{a,b}\in P' : a\neq b\}$
we have the counterpart of $S_{c_1,c_1}$, which does not attain the minimum.
So this is also a contradiction to our choice of $c_1$ and~$c_2$,
and our Pl\"ucker relation is proved in this case, completing the proof.
\end{proof}

\section{Global structure of matroids over a Dedekind domain}\label{sec:global}
Throughout this section $R$ will be a Dedekind domain. 
Let us recall that given a $R-$module $N$, by $\det(N)$ we will denote its class in the Picard group, as defined in Section \ref{sec:Dedekind}.
Understanding the local ring case, we can now give a necessary and sufficient 
condition for which pairs of modules can occur in condition~(M1).

\begin{proposition}\label{p:Dedekind M1}
Let $N$ and~$N'$ be f.g.\ $R$-modules.  
There exists a surjection $N\to N'$ with cyclic kernel if and only if
there exists such a surjection $N_{\mf m}\to N'_{\mf m}$
after localizing at each maximal prime $\mf m$ of~$R$, and 
\begin{itemize}
\item if $\rk(N)-\rk(N')=0$ then $\det(\projpart{N}) = \det(\projpart{N'})$, whereas
\item if $\rk(N)-\rk(N')=1$ then $\det(N) = \det(N')$.
\end{itemize}
\end{proposition}

To test whether surjections exist in the localizations,
we have the criterion in Proposition~\ref{p:DVR M1}.

\begin{proof}
\noindent\emph{Necessity.}  Localization is a base change
so preserves condition~(M1).  
If $\rk(N) = \rk(N')$, then the kernel of $N\to N'$
is contained in $\tors{N}$, so that $\projpart{N}\cong\projpart{N'}$,
and so their classes are equal.
If $\rk(N) = \rk(N')+1$
then the kernel of $N\to N'$
must be a cyclic rank~1 $R$-module, which up to isomorphism is $R$.
Therefore $\det(N) = \det(R)+\det(N') = \det(N')$
by the definition of $K_0(R)$. 

\noindent\emph{Sufficiency.}  Note first that $\rk(N)-\rk(N')\in\{0,1\}$,
because the same is true in every localization.  
Let us suppose that $\rk(N)=\rk(N')$.  Then 
$\det(\projpart{N})=\det(\projpart{N'})$ implies 
$\projpart{N}\cong\projpart{N'}$, by Proposition~\ref{p:Dedekind structure}.
Moreover, $\tors{N}$ and $\tors{N'}$ are the direct sums of their localizations.
The kernel of each of the given maps $N_{\mf m}\to N'_{\mf m}$ 
is contained in the torsion $\tors{(N_{\mf m})} = (\tors{N})_{\mf m}$,
so a map $(\tors{N})_{\mf m}\to (\tors{N'})_{\mf m}$ is induced.
The direct sum of all these maps is a map $\tors{N}\to\tors{N'}$
which is still a surjection; its kernel is a sum of cyclic modules
with disjoint support, which is still cyclic.  
Taking the direct sum with an isomorphism $\projpart{N}\to\projpart{N'}$
yields the requisite map $N\to N'$.  

Now suppose that $\rk(N) = \rk(N') +1$.  
In this case, construct a set function $M:\mc B(1)\to\fgMod{R}$
so that $M(\emptyset) = \tors{N'}$ 
and $M(1) = \tors{N}$.  
Note that $M^*(\emptyset)$ and $M^*(1)$ both have
rank~0 and therefore trivial projective part.  
Moreover, there exist localized surjections with
cyclic kernel $M(\emptyset)_{\mf m}\to M(1)_{\mf m}$
for each~$\mf m$.
This is by Proposition~\ref{p:duality}(c)
(or because localization is flat),
because $M(\emptyset)_{\mf m}$ and $M(1)_{\mf m}$
are the modules of a one-element matroid over $R_{\mf m}$, 
the dual of the matroid built from $N_{\mf m}\to N'_{\mf m}$.  

Using the previous case, 
there exists a surjection with cyclic kernel $M(\emptyset)\to M(1)$.  
That is, $M$ is a matroid over~$R$. 
By Proposition~\ref{p:duality}(a), 
dualizing this matroid and taking the direct sum with the empty matroid
for $\projpart{N'}$ yields a one-matroid over $R$ 
whose objects are $N$ and $N'$.  There thus exists a surjection
with cyclic kernel $N\to N'$.  
\end{proof}

For a complete description of the structure of matroids over~$R$
we must of course treat the axiom~(M).  It turns out there are no
(ring-theoretically) global conditions on such squares, and thus on matroids over~$R$,
further to those imposed by condition~(M1).  

\begin{proposition}\label{p:Dedekind M2}
Let $M(\emptyset)$, $M(1)$, $M(2)$, and $M(12)$ be f.g.\ $R$-modules.
There exist four surjections with cyclic kernels forming a pushout square
\[\xymatrix{
\ar@{}[dr]|{\mbox{\LARGE$\lrcorner$}}
M(\emptyset)\ar[r]\ar[d] & M(1)\ar[d] \\
M(2)\ar[r] & M(12)
}\]
if and only if the same is true after localizing at each maximal prime $\mf m$,
and the source and target of each map satisfy Proposition~\ref{p:Dedekind M1}.
\end{proposition}

\begin{proof}
\noindent\emph{Necessity.} Trivial in view of Proposition~\ref{p:Dedekind M1}
and the fact that pushout squares localize to pushout squares.

\noindent\emph{Sufficiency.} Fix a pushout square for each localization;
label its maps as follows.
\[\xymatrix{
\ar@{}[dr]|{\mbox{\LARGE$\lrcorner$}}
M(\emptyset)_{\mf m}\ar^{f_{\mf m}}[r]\ar_{g_{\mf m}}[d] & M(1)_{\mf m}\ar^{g'_{\mf m}}[d] \\
M(2)_{\mf m}\ar_{f'_{\mf m}}[r] & M(12)_{\mf m}
}\]
It is enough to construct two maps
$M(\emptyset)\to M(1)$ and $M(\emptyset)\to M(2)$ which localize correctly
everywhere, for then we may choose $M(12)$ to be their pushout, 
since pushouts localize to pushouts.  

Suppose first that one of $M(1)$ and $M(2)$ has the same rank as $M(\emptyset)$,
without loss of generality that $M(2)$ does.  By Proposition~\ref{p:Dedekind M1},
we may construct a map $\phi:M(\emptyset)\to M(1)$ so that there exist isomorphisms
$i_{\mf m}(A)$ for each prime $\mf m$ and $A=\emptyset, 1$ making the squares
\[\xymatrix{
M(\emptyset)_{\mf m}\ar[r]^{\phi_{\mf m}}\ar[d]_{i_{\mf m}(\emptyset)} 
& M(1)_{\mf m}\ar[d]^{i_{\mf m}(1)} \\
M(\emptyset)_{\mf m}\ar[r]_{f_{\mf m}} 
& M(1)_{\mf m}
}\]
commute.
Now, by the proof of Proposition~\ref{p:Dedekind M1}, we construct $M(0)\to M(2)$
as the direct sum of the restriction of the given $g_{\mf m}$ 
to the torsion submodule of $M(\emptyset)_{\mf m}$, and the identity map 
$\projpart{M(0)}\to\projpart{M(0)}$.  It changes nothing to
precompose each of these restrictions of $g_{\mf m}$ by the corresponding
$i_{\mf m}(\emptyset)$.  Doing this yields a commutative square
\[\xymatrix{
M(\emptyset)_{\mf m}\ar[r]^{\psi_{\mf m}}\ar[d]_{i_{\mf m}(\emptyset)} 
& M(2)_{\mf m}\ar[d] \\
M(\emptyset)_{\mf m}\ar[r]_{g_{\mf m}} 
& M(2)_{\mf m}
}\]
and pasting this square to the last one shows that we have constructed
the two maps $M(\emptyset)\to M(1)$ and $M(\emptyset)\to M(2)$ 
which localize as desired.

The remaining case is the one in which the ranks of $M(1)$ and $M(2)$
are both less than that of $M(\emptyset)$.  In this case, like the second 
case of Proposition~\ref{p:Dedekind M1},
we will proceed via dualization, and then via a similar argument.  
In brief, we may first construct a map
$M^*(2)\to M^*(12)$ which localizes correctly, up to intertwining with some isomorphisms.
Then since the map $M^*(1)\to M^*(12)$ doesn't involve a rank drop,
we may construct it as a direct sum of localizations on the torsion parts
using the same isomorphisms.  This gives us a diagram 
$M^*(1)\to M^*(12)\gets M^*(2)$ which localizes correctly at every maximal prime.
Finally, we may temporarily insert any suitable
module in place of $M^*(\emptyset)$, for instance the pullback,
and then dualize the resulting matroid over~$R$.
Discarding the ersatz $M(12)$ gives us maps 
$M(2)\gets M(\emptyset)\to M(1)$
which localize correctly, as desired.  
\end{proof}

\subsection{Quasi-arithmetic matroids}

If $M$ is a matroid over~$\mathbb{Z}$, then we can define a corank function of $M$
as the corank function of the generic matroid $M\otimes_\Z\Q$ described above, that is
$$\cork(A)=\dim (\projpart{M(A)}).$$
We also define
$$m(A)\doteq |\tors{M(A)}|.$$

\begin{corollary}\label{c:arithmetic}
The triple $(E, \cork, m)$ is a \emph{quasi-arithmetic matroid}, i.e $m$ satisfies the following properties:
\begin{itemize}
    \item[(A1)] Let be $A\subseteq E$ and $b\in E$;
    if $b$ is dependent on $A$, then
    $m(A\cup\{b\})$ divides $m(A)$;
    otherwise
    $m(A)$ divides $m(A\cup\{b\})$;
    \item[(A2b)] if $A\subseteq B\subseteq E$ and $B$ is a disjoint union $B=A\cup F\cup T$
    such that for all $A\subseteq C\subseteq B$  we have $\rk(C)=\rk(A)+|C\cap F|$, then
    $$
    m(A)\cdot m(B) = m(A\cup F)\cdot m(A\cup T).
    $$
\end{itemize}

Furthermore it satisfies the following property:
\begin{itemize}
    \item[(A2a)] if $A,B\subseteq E$ and $\rk(A\cup B)+\rk(A\cap B)=
    \rk(A)+\rk(B)$, then $m(A)\cdot m(B)$ divides $m(A\cup B)\cdot m(A\cap
    B)$
\end{itemize}
\end{corollary}

\begin{proof}
Since $\Pic(\mathbb{Z})$ is trivial, this is immediate from Propositions \ref{p:DVR M1}, \ref{p:DVR M2}, \ref{p:Dedekind M1} and~\ref{p:Dedekind M2}. 
\end{proof}

This corollary establishes that
matroids over~$\mathbb{Z}$ recover many of the essential features of
the second author's theory of \emph{arithmetic matroids} from \cite{D'Adderio-Moci}.

\begin{remark}\label{rem:quasi}
An arithmetic matroid is an object satisfying
all the axioms of a quasi-arithmetic matroid plus a further one,
namely the positivity property~(P) of~\cite{Branden-Moci}.
The axioms of quasi-algebraic matroids are arithmetic ones, pertaining to integer divisibility,
whereas (P) has geometric motivation.  
To be precise, (P) is included as an axiom
in order to force positivity of the arithmetic Tutte polynomial $M_A(x,y)$.
Its geometric nature is that the numbers whose positivity it demands are,
in the realizable case, numbers of components of certain strata in
the corresponding toric arrangement.  
Its coefficients also have 
two natural but nontrivial combinatorial interpretations \cite{D'Adderio-Moci, Branden-Moci}.

The additional property (A2a) appeared in an earlier choice of the axioms \cite[first arXiv version]{D'Adderio-Moci}.
\end{remark}

In fact quasi-arithmetic matroids and matroids over $\mathbb{Z}$ are not truly equivalent, in that the information contained in the latter is richer:
it retains isomorphism classes of torsion groups, not just their cardinalities.

\section{The Tutte-Grothendieck ring}\label{sec:Grothendieck}

In this section we continue to let $R$ be a Dedekind domain.
All matroids over~$R$ in this section are essential.  
The word ``matroid'' will mean ``matroid over~$R$'' from here
through the end of the proof of Lemma~\ref{lem:G molecule}, 
except when we speak of a generic matroid.

As we defined the operations of deletion and contraction in Section~\ref{sec:definition}, 
any element may be deleted or contracted.  However,
if $a\in E$ is a generic coloop, then $M\setminus a$ is not essential,
so we will disallow these deletions here.
Dually, we will exclude the case of contracting a generic loop.

Essentially following Brylawski~\cite{Brylawski},
define the \emph{Tutte-Grothendieck ring} of matroids over $R$,
which we here denote $K(\Mat{R})$, to be
the ring whose underlying abelian group generated by a symbol $\T_M$ for each unlabelled
essential matroid $M$ over~$R$ with nonempty ground set, modulo the relations
\[ \T_M = \T_{M\setminus a} + \T_{M/a} \]
whenever $a$ is not a generic loop or coloop;
and whose multiplication is given by linear extension from the relation
\[\T_M\cdot \T_{M'} = \T_{M \oplus M'}.\]

By ``unlabelled'', we mean that we consider two matroids $M$ and~$M'$ over~$R$ to be
identical if there is a bijection $\sigma:E\stackrel\sim\to E'$ of their ground sets such
that $M(A)\cong M'(\sigma(A))$ for each subset $A$ of~$E$.

The ring $K(\Mat{R})$ turns out to be best understood 
in terms of the monoid ring of the monoid of $R$-modules under direct sum,
as in Theorem~\ref{th:Grothendieck} below.  This however
only identifies a ring which $K(\Mat{R})$ injects into;
the precise description of the image is given in
Corollary~\ref{cor:what the ring is exactly}.

Define $\Z[\fgMod{R}]$ to be the ring with a $\Z$-linear basis 
$\{u^N\}$ with an element $u^N$ for each f.g.\ $R$-module $N$ up to isomorphism,
and product given by $u^Nu^{N'} = u^{N\oplus N'}$.

\begin{theorem}\label{th:Grothendieck}
The Tutte-Grothendieck ring injects into 
\[\Z[\fgMod{R}]\otimes \Z[\fgMod{R}],\]
in such a way that for every matroid $M$ over $R$,
\begin{equation}\label{eq:T}
\T_M \mapsto \sum_{A\subseteq E} X^{M(A)} Y^{M^*(E\setminus A)}
\end{equation}
where $\{X^N\}$ and $\{Y^N\}$ are the respective bases of the two 
tensor factors $\Z[\fgMod{R}]$.
\end{theorem}

As a point of notation, we will allow ourselves the abbreviation $(XY)^N$ for $X^NY^N$.

We immediately compare Theorem~\ref{th:Grothendieck} with the case of
matroids over a field, where the Tutte-Grothendieck invariant is
the familiar Tutte polynomial $\T_M$; 
in Section~\ref{ssec:arithmetic Tutte}
we will relate it to other known invariants.
If $R$ is a field, then $\Z[\fgMod{R}]$ is the univariate polynomial
ring $\Z[u]$, and then $\Z[\fgMod{R}]\otimes\Z[\fgMod{R}]$ 
is, appropriately, a bivariate polynomial ring.  
If we call the generators of the two tensor factors $x-1$ and~$y-1$
rather than $X$ and~$Y$,
then equation~\eqref{eq:T} in fact gives the classical Tutte polynomial,
since $\dim M(A)$ is the corank of~$A$ and $\dim M^*(E\setminus A)$ is its nullity.

\begin{remark}
We have excluded empty matroids from the definition
of~$K(\Mat{R})$ because there are no linear relations relating them
to matroids with nonempty ground set: the unique element
in a matroid on one element, whence one might get a relation, must be a loop or coloop.
Thus, constructing the Tutte-Grothendieck ring in the presence of 
zero-element matroids would yield a ring which would,
in a minimal fashion, fail to be a domain or to
inject into $\Z[\fgMod{R}]\otimes \Z[\fgMod{R}]$.
Applying equation~\eqref{eq:T} to 
a matroid $M$ on zero elements yields the monomial
$X^{M(\emptyset)}Y^{M^*(\emptyset)}$.
But sums of such monomials can also be achieved as sums of polynomials
$\T_M$ for nonempty $M$, and these cannot be equal in $K(\Mat{R})$.
However, if classes $\T_M$ for empty matroids are defined via 
equation~\eqref{eq:T}, these classes behave correctly under 
the multiplication of~$K(\Mat{R})$.
%
\end{remark}

Since decomposing a matroid $M$ over a ring into $M\setminus i$ and $M/i$ is
not a unique decomposition in the sense of~\cite{Brylawski},
and the irreducibles for direct sum are not all single-element matroids,
Theorem~\ref{th:Grothendieck} does not follow directly from the 
bidecomposition methods of~\cite{Brylawski}, and we must
prove it by hand.

For the proof it will be useful to have some explicit understanding
of the ring $\Z[\fgMod{R}]\otimes \Z[\fgMod{R}]$.
Proposition~\ref{p:Dedekind structure}
implies that $\Z[\fgMod{R}]$ has one generator
$u^{R/\mf m^k}$ for each maximal ideal $\mf m$ and integer $k>0$,
and no relations involving these, together with one generator $u^P$ for each rank~1
projective module~$P$, among which there are many relations.
Indeed, the subring $\Z[u^P]$ embeds in the group ring of the Picard
group with one more variable $u^R$ adjoined.  

\begin{proof}[Proof of Theorem~\ref{th:Grothendieck}]
To be concise, let $S$ be the ring $\Z[\fgMod{R}]\otimes \Z[\fgMod{R}]$.
To keep distinct the objects which we have not yet proven isomorphic,
let $[M]$ represent the class of $M$ in $K(\Mat{R})$, reserving $\T_M$ for the
element of~$S$ defined in~\eqref{eq:T}.  

Consider the map $\T : K(\Mat{R})\to S$ given by 
$\T([M]) = \T_M$.  
We have that $\T$ is a homomorphism of rings,
because the deletion-contraction relations
and multiplicativity relations hold among the various $\T_M$.  
Both of these are straighforward to check, and
correspond to easy operations on equation~\eqref{eq:T}.
The deletion-contraction relation on an element~$a$ 
is proved by splitting the sum into one sum containing
the terms with $a\not\in A$ and another containing the terms with $a\in A$.
Multiplicativity under direct sum is proved by expanding the product of 
equation \eqref{eq:T} for $M$ and~$M'$, and collecting into a single
sum over $A\amalg A'\subseteq E\amalg E'$.  


With that, we come to the involved part of the proof,
which is to show $\T$ an injection.  
Our approach will be to construct a family 
$\mc I$ of matroids $M$ whose polynomials $\T_M$
are linearly independent in~$S$, 
and use deletion-contraction
relations to expand every matroid in terms of $\mc I$.
This will allow every linear relation among the $\T_M$ to
be lifted to a relation among the $[M]$ via expansion
in terms of~$\mc I$, proving injectivity.
We will moreover be able to conclude that 
the image of~$\T$ is the span of the images of the matroids in~$\mc I$.

As we use the deletion-contraction relations,
we will make frequent use of induction on the size of the ground set.
In fact, our main technique will be to embed a matroid $M$ as a minor of another, $M'$, 
and then relate $M$ to another minor of $M'$ of the same size plus a collection of
smaller minors.  
But if the ground set has size~1, this will not be as useful: the unique element
of a 1-element matroid is necessarily either a loop or a coloop, 
hence we cannot get construct a deletion-contraction relation 
involving a smaller matroid as a minor.  This will be our base case,
and require a different argument.  
We have broken out the arguments expanding these matroids in terms of $\mc I$
into two lemmas, Lemma~\ref{lem:G atom} and Lemma~\ref{lem:G molecule}.

The following construction is relevant to both cases.  
Linearly extend the divisibility relation on the ideals of~$R$
to a total order $\leq$ such that for ideals $I,J,K$,
$I\leq J$ implies $IK\leq JK$.
For each class $\mc E\in\Pic(R)$,
let $N_{\mc E}$ equal $R/I$, where $I$ is the $\leq$-least ideal of~$R$ 
whose determinant is $\mc E^{-1}$.
This produces a fixed cyclic torsion module $N_{\mc E}$ representing 
each class $\mc E\in\Pic(R)$.  Note that every submodule $N'$ of $N_{\mc E}$
is also the representative of its own class, $N' = N_{[N']}$.
Define the single-element matroid
$L_{\mc E}$ by $L_{\mc E}(\emptyset) = N_{\mc E}$ and 
$L_{\mc E}(1) = 0$.
The dual matroid $L^*_{\mc E}$ therefore has
$L^*_{\mc E}(\emptyset) = R$ and $L^*_{\mc E}(1) = N_{\mc E}$.  
($L$ and $L^*$ can be taken to stand for ``loop'' and ``coloop''.)
Also, let $\emptyset_N$ be the empty matroid associated to
a torsion $R$-module $N$.

For a torsion module $N$, define a second sort of loop $K_N$
by taking $K_N(\emptyset) = N$ and $K_N(1)$ to be the quotient of~$N$
by its largest invariant factor.  

We construct the set $\mc I$ as follows. 
\begin{equation}\label{eq:G I}\begin{aligned}
\mc I &= \{K_N : \mbox{$N$ is torsion}\}
\\&\cup \{\emptyset_N\oplus L_{\mc E}\oplus L_0^{\oplus a} : 
\mbox{$N$ is torsion},\, \mc E\in\Pic(R), a\geq0\}
\\&\cup \{\emptyset_N\oplus L^*_{\mc F}\oplus (L^*_0)^{\oplus b} : 
\mbox{$N$ is torsion},\, \mc F\in\Pic(R), b\geq0\}
\\&\cup \{\emptyset_N\oplus L_{\mc E}\oplus L_0^{\oplus a}\oplus L^*_{\mc F}\oplus (L^*_0)^{\oplus b} : 
\mbox{$N$ is torsion},\, \mc E,\mc F\in\Pic(R), a,b\geq0\}
\end{aligned}\end{equation}

To analyze linear relations in~$\mc I$, we
give the ring $S$ a monomial order wherein, if $P$ and $Q$ are rank~1
projective modules, then $X^P$ is greater than $Y^Q$, which in turn is
greater than $X^N$ or $Y^N$ for any torsion module $N$.
  
Then if $M$ is a matroid with a unique basis, as all the matroids in~$\mc I$ are,
the initial term of $\T_M$ is the term contributed to the sum in~\eqref{eq:T}
by the complement of the unique basis of~$M$.  For the matroid
\[M=\emptyset_N\oplus L_{\mc E}\oplus L_0^{\oplus a}\oplus L^*_{\mc F}\oplus (L^*_0)^{\oplus b}\]
the complement of the unique basis is sent to $N\oplus P\oplus R^b$,
where $P$ is the rank~1 projective module whose determinant is $\mc F\in\Pic(R)$.
For the dual of this matroid, the analogous module is $N\oplus Q\oplus R^a$
where $Q$ is the rank~1 projective whose determinant is $\mc E$.
Therefore the initial term of $\T_M$ is 
$X^{N\oplus P\oplus R^b}Y^{N\oplus Q\oplus R^a}$.
Similarly, if instead of~$M$ we had taken one of the matroids from the two 
previous lines in~$\mc I$, the initial term
would be just $X^{N\oplus P\oplus R^b}$ or $Y^{N\oplus Q\oplus R^a}$, respectively.  
All the monomials in these three classes are distinct.

Finally, the initial term of $\T_{K_N}$ is $Y^Q$ times monomials
corresponding to torsion modules, for some rank~1 projective $Q$.  
It follows that any nontrivial $\Z$-linear relation 
among the classes of elements of $\mc I$ 
may contain only these matroids and others of smaller leading terms:
that is, it may involve only one-element matroids whose unique element is a loop.

Temporarily let $\mc I_1$ be the set of the matroids $K_N$, and
$\mc I_2$ the set of matroids of form $\emptyset_N\oplus L_{\mc E}$,
so that together every matroid in~$\mc I$ whose unique element is a loop
is in $\mc I_1$ or $\mc I_2$.
Suppose there was a nontrivial linear dependence among the classes of these matroids.
(The sets $\mc I_1$ and $\mc I_2$ share some elements, but since we have taken their union as \emph{sets}
this is not a problem.)  The class of each matroid in $\mc I_1\cup\mc I_2$
is of the form 
$(XY)^N + Y^Q(XY)^{N'}$
where $N$ and~$N'$ are
torsion $R$-modules, and $Q$ is a rank 1 projective module.
Moreover, there is only one element of $\mc I_1$ 
and one of~$\mc I_2$ with a given value of~$N$.  Therefore, if there is any linear relation, 
there must be a minimal one of the form
\begin{equation}\label{eq:G I relation}
\sum_{j=1}^k[M_{1,j}] - \sum_{j=1}^k [M_{2,j}]=0
\end{equation}
where $M_{i,j}\in\mc I_i$, all $M_{i,j}(\emptyset)$ have the same determinant in~$\Pic(R)$,
and $M_{1,j}(\emptyset) = M_{2,j}(\emptyset)$.  The equality also implies that
the product of the annihilators in~$R$ of the kernels of $M_{1,j}(\emptyset)\to M_{1,j}(1)$
equals the corresponding product of annihilators for the kernels of $M_{2,j}(\emptyset)\to M_{2,j}(1)$.
All of these annihilators have the same determinant.
The latter product is $I^k$, where $I$ is the annihilator of $N_{\mc E}$.
Therefore, at least one of the ideals in the former product must be less
than or equal to $I$ in $\leq$ order.  But $I$ is the $\leq$-minimal ideal
of its class, and so these ideals must all equal $I$, so that all
coefficients on the left side of \eqref{eq:G I relation} are zero
and the relation is trivial.  Thus $\mc I$ is dependent, as claimed.
\end{proof}

We may now describe the image of~$K(\Mat{R})$ within the ring $\Z[\fgMod{R}]\otimes \Z[\fgMod{R}]$.
Two constraints on the monomials that may appear can be extracted from Corollary~\ref{cor:dual}.

\begin{corollary}\label{cor:constraints on T}\ 
\begin{enumerate}\renewcommand{\labelenumi}{(\alph{enumi})}
\item If $X^NY^{N'}$ is a term of~$\T_M$, then $\tors{N} = \tors{N'}$.
\item Consider the ring homomorphism 
\[\cl:\Z[\fgMod{R}]\otimes \Z[\fgMod{R}]\to \Z[u^{\mc E} : \mc E\in\Pic(R)]\]
given by 
$\cl(X^N) = u^{[N] - [\tors{N}]}$ and $\cl(Y^N) = u^{[N]}$.
Then $\cl(\T_M)$ is a scalar multiple of $u^{\cl(M)}$.
\end{enumerate}
\end{corollary}

\begin{proof}
Part~(a) is immediate from the preservation of torsion parts under duality in Corollary~\ref{cor:dual},
and part~(b) from the equality
\[\cl(M) = \det(M(A)) + \det(M^*(E\setminus A))- \det(\tors{M(A)}).\qedhere\]
\end{proof}

Examination of the classes of matroids in~$\mc I$
in the proof of Theorem~\ref{th:Grothendieck}
shows that they span the subring of polynomials
compatible with Corollary~\ref{cor:constraints on T}(a).
Thus we conclude the following.
\begin{corollary}\label{cor:what the ring is exactly}
If empty matroids are included,
the ring $\K(\Mat{R})$ is the subring of $\Z[\fgMod{R}]\otimes \Z[\fgMod{R}]$
generated by the symbols $X^P$ and $Y^P$ as $P$ ranges over rank~1 projective modules,
and $(XY)^N$ as $N$ ranges over torsion modules.
\end{corollary}

We now set out on the main substance of the proof of Theorem~\ref{th:Grothendieck}.
The following subsidiary lemma will afford us useful flexibility for 
manipulating single-element matroids in proving Lemma~\ref{lem:G atom}.

\begin{lemma}\label{lem:disjoint}
Let $\mc E$ be a class in~$\Pic(R)$, and $P$ a finite set of maximal primes of~$R$.
There exists a cyclic torsion $R$-module $N$ whose support is disjoint from $P$
such that $\det(N) = \mc E$.
\end{lemma}

\begin{proof}
This is a restatement of a standard lemma on ideal factorizations,
see e.g.\ \cite[Corollary 4.9]{Conrad-notes}.
In the notation of that corollary, we let $\mf a$ be any ideal of determinant $\mc E$
and let $\mf b$ be the product of the members of~$P$, and the module we seek is $N=R/\mf c$.
\end{proof}

\begin{lemma}\label{lem:G atom}
If $M$ is a one-element matroid satisfying any one
of the following, the class $[M]\in K(\Mat{R})$
lies in the span of the classes of matroids in the set $\mc I$ of~\eqref{eq:G I}.
\begin{enumerate}\renewcommand{\labelenumi}{(\alph{enumi})}
\item $M(\emptyset)=P\oplus N$ and $M(1)=N\oplus C$,
where $P$ is rank~1 projective, $N$ is torsion and $C$ is cyclic, 
and the supports of $N$ and $C$ are disjoint.
\item $M(\emptyset)=N\oplus C$ and $M(1) = N$,
where $N$ is torsion and $C$ is cyclic, 
and the supports of $N$ and $C$ are disjoint.
\item $M(\emptyset) = P\oplus N'$, $M(1) = N$, where 
$P$ is rank~1 projective, $N$ is torsion, and
$N'$ is the quotient of $N$ by its largest invariant factor.
\item either $M$ or $M^*$ sends
$\emptyset$ to $P\oplus T\oplus N'$ and $\{1\}$ to $T\oplus N$, where
$P$ is rank~1 projective, $N$ is torsion,
$N'$ is the quotient of $N$ by its largest invariant factor, and
the support of $T$ is disjoint from that of~$N$.
\item either $M$ or $M^*$ sends
$\emptyset$ to $N\oplus T$ and $\{1\}$ to $T$, 
where $N$ is cyclic and $T$ is torsion.
\item $M$ is any one-element matroid.
\end{enumerate}
\end{lemma}

\begin{proof}
A matroid $M'$ on two elements whose generic matroid
is the uniform matroid $U_{1,2}$
gives rise to a linear relation among
its four one-element minors, 
\begin{equation}\label{eq:G2}
[M'\setminus 1] + [M'/1] = [M'] = [M'\setminus 2] + [M'/2].
\end{equation}
We will use this to prove the cases of the lemma sequentially,
reducing each to a linear combination of matroids in $\mc I$
and matroids in previous cases.
For visibility we will specify these $M'$ by drawing the
commutative square
\[\xymatrix{
M'(\emptyset)\ar[r]\ar[d] & M'(1)\ar[d] \\
M'(2)\ar[r] & M'(12)
}.\]
In each case $M'$ can be checked to be a matroid by Proposition~\ref{p:Dedekind M2}.
The non-local conditions reduce to checking that $\det(M'(1)) = \det(M'(2))$.

\noindent\emph{To (a).}
Let $N'$ be the quotient of~$N$ by its largest invariant factor.
First, suppose that the support of~$C$ is disjoint from the supports of $N$ and~$L_{[C]}(1)$.
In that case, the following square specifies a matroid $M'$.  
The modules $N\oplus C$ and $N'\oplus L_{C}(1)$
have the same determinant by construction.  And because of 
the assumption on supports, modulo each maximal ideal $\mf m$, 
either the top map has kernel $R_{\mf m}$ and the right one is trivial,
or the same is true of the left and bottom maps respectively.
This ensures that the localizations of the square are pushouts.
\[\xymatrix{
P\oplus N \ar[r]\ar[d] & N\oplus L_{[C]}(1) \ar[d] \\
N\oplus C \ar[r] & N'
}\]
The left minor $M'\setminus 1$ is the matroid we are interested in; 
the bottom minor $M'/2$ and 
the right minor $M'/1$ are both among the matroids $K_N$ in~$\mc I$;
and the top minor $M'\setminus 2$ is among the $\emptyset_N\oplus L^*_{\mc F}$.
So the relation~\eqref{eq:G2} proves the result in this case.

Next, if we lack the support assumptions on~$C$,
we are assured the existence of a cyclic module $C'$
of support disjoint from $L_{[C]}(1)$ and~$N$, with the same determinant as~$C$,
by Lemma~\ref{lem:disjoint}.  In this case,
we repeat the argument with the following square,
which can similarly be checked to give a matroid $M''$.
\[\xymatrix{
P\oplus N \ar[r]\ar[d] & N\oplus C' \ar[d] \\
N\oplus C \ar[r] & N'
}\]
Now $M''\setminus 1$ is the matroid of interest, 
the minors $M'/2$ and $M'/1$ are among the matroids $K_N$,
and $M'\setminus 2$ is in the span of~$\mc I$ by the last paragraph.
Therefore, using \eqref{eq:G2} again, we have proved case~(a).

\noindent\emph{To (b).}  
Let $R/I$ be the largest invariant factor of~$N$. 
Let $J$ be a nonzero ideal contained in~$I$
chosen so that $\det(J) = -\det(C)$.
and the supports of $R/J$ and~$L_{[C]}(1)$ are disjoint; this exists by Lemma~\ref{lem:disjoint}.
Then $R/J$ is the largest invariant factor of $N\oplus R/J$.
Having done this, both of the following squares give matroids,
where $P$ is a suitably chosen rank~1 projective module.
\[
M': 
\xymatrix{
P\oplus N \ar[r]\ar[d] & N\oplus L_{[C]}(1) \ar[d] \\
N\oplus R/J \ar[r] & N
}
\qquad
M'':
\xymatrix{
P\oplus N \ar[r]\ar[d] & N\oplus C \ar[d] \\
N\oplus R/J \ar[r] & N
}
\]
Subtracting the relation~\eqref{eq:G2} for the two matroids, 
we express the class of $M''/1$, which is the matroid of interest,
as a linear combination of the classes of $M''\setminus 2$, $M'/1$, and~$M'\setminus 2$.
But, of these, $M''\setminus 2$ is one of the matroids appearing
in part~(a), $M'/1$ is of form $\emptyset_N\oplus L_{\mc E}$, 
and $M'\setminus 2$ is of form $\emptyset_N\oplus L^*_{\mc F}$.  
This proves case~(b).

\noindent\emph{To (c).}  
Use Lemma~\ref{lem:disjoint} to produce a cyclic module $C$
whose determinant is $\det(N)-\det(N')$ and whose support
is disjoint from that of~$N$.  Then the following square gives a matroid $M'$.
\[\xymatrix{
P\oplus N' \ar[r]\ar[d] & N'\oplus C \ar[d] \\
N \ar[r] & N'
}\]
Here, $M'/2$ is of form $\emptyset_N\oplus L_{\mc E}$, and $M'/1$ is
covered by case~(b) of the lemma, and $M'\setminus 2$ is covered by case~(a).
So \eqref{eq:G2} proves case~(c).

At this point, we pause to take note that the matroids 
of form $\emptyset_N\oplus L_{\mc E}$ and $\emptyset_N\oplus L^*_{\mc E}$
and the matroids encompassed by this last case (c)
are the duals of the $K_N$.  These are all the matroids in~$\mc I$ of one element.
Moreover, in equation \eqref{eq:G2}, dualizing $M'$ dualizes the four minors.
Therefore, for the rest of this proof, arguing that a class $[M]$
is in the linear span of classes of matroids in~$\mc I$ will imply
the same for the class $[M^*]$ of the dual.

\noindent\emph{To (d).}
As stated just above, it is sufficient to treat the case where
$M$ is as described, not its dual.
For this we use induction on the number of invariant factors of
the torsion module $T$.
If it has none, it is the zero module and we are in case~(a).
Otherwise, let $T'$ be the quotient of~$T$ by its largest invariant factor.  

Let $N''$ be the quotient of $N'$ by its largest invariant factor.
Lemma~\ref{lem:disjoint} gives cyclic modules $C$ and~$D$
whose determinants take the necessary values,
and whose supports are disjoint from the supports of other appearing modules as necessary,
in order for the following squares to specify matroids $M'$ and~$M''$.  
\[
M':
\xymatrix{
P\oplus N'\oplus T \ar[r]\ar[d] & N'\oplus T\oplus C \ar[d] \\
N\oplus T \ar[r] & N'\oplus T'
}
\qquad
M'':
\xymatrix{
P\oplus N''\oplus T' \ar[r]\ar[d] & N'\oplus T\oplus C \ar[d] \\
N'\oplus T'\oplus D \ar[r] & N'\oplus T'
}\]
Here $M'/1$ equals $M''/1$, and these can be cancelled out
of the two corresponding invocations of \eqref{eq:G2},
leaving a linear relation among the six other minors.
Of these, $M'\setminus 1$ is the matroid of interest.
$M''\setminus 1$ is the matroid to which we will apply
the induction hypothesis: when applying it we take
the new module $T$ to be $T'$, which has one invariant factor fewer 
than (the old) $T$, and the new modules $N'\to N$ to be
$N''\to N'\oplus D$.  
The remaining minors are dealt with: $M'/2$ is among the $K_N$,
$M''/2$ is dealt with in case~(b) of this lemma, 
$M'\setminus 2$ in case~(a), and $M''\setminus 2$ in case~(c). 
Therefore the induction goes through and we have proved case~(d).

\noindent\emph{To (e).}
Again we may assume $M$ (not $M^*$) is as described.
We use induction on the maximum $k$ such that, for some maximal prime $\mf m$
contained in the support of $N$,
there are $k$ cyclic summands of~$T_{\mf m}$ longer than $N_{\mf m}$.  
If $k=0$, then $N$ is the largest invariant factor of 
the part of~$N\oplus T$ with the same support,
and $M$ falls under case~(d).  

Otherwise, let $F$ be the direct sum of the localizations
of the largest invariant factor of~$T$ of length
exceeding the corresponding localization of~$N$,
and let $T'$ be $T/F$.
Lemma~\ref{lem:disjoint} provides cyclic modules $C$ and~$D$
so that the following squares are matroids.
(For the top maps of the squares, this is where the fact
that $\dim_{R/\mf m} F_{\mf m}\geq\dim_{R/\mf m} N_{\mf m}$ is used.)
\[
M':
\xymatrix{
P\oplus N \ar[r]\ar[d] & T\oplus C \ar[d] \\
N\oplus \T \ar[r] & T
}
\qquad
M'':
\xymatrix{
P\oplus N \ar[r]\ar[d] & T\oplus C \ar[d] \\
N\oplus \T'\oplus D \ar[r] & T'
}
\]
Again, $M'/2$ equals $M''/2$, and two invocations of \eqref{eq:G2}
give a linear relation among the remaining minors.  
Of these $M'\setminus 2$ is the matroid of interest,
$M''\setminus 2$ is covered by the inductive hypothesis,
and all of $M'/1$, $M''/1$, $M'\setminus 1$, and $M''\setminus 1$ 
fall under case~(d).  This proves case~(e).

\noindent\emph{To (f).}
Either $M$ or $M^*$ is of global rank~0, and we may assume it is $M$.
Here we use one further induction.
Let $k$ be the maximum, over maximal primes $\mf m$, 
of the number of times $01$ or~$10$ appear as substrings
of the sequence $d_\bullet(\phi_{\mf m})$ associated in Section~\ref{sec:DVR} to
the map $\phi_{\mf m}$ in the localized matroid $M\otimes R_{\mf m}$.
As a base case, if $k\leq 1$, then each $\phi_{\mf m}$, and therefore
$\phi$, is a quotient by a cyclic summand; this is case~(e).

Otherwise, let $N$ be the quotient of~$M(\emptyset)$
by its largest invariant factor.
With $C$ provided by Lemma~\ref{lem:disjoint} as usual,
we have a matroid $M'$ given by
\[
\xymatrix{
P\oplus N \ar[r]\ar[d] & M(1)\oplus C \ar[d] \\
M(\emptyset) \ar[r] & M(1)
}
\]
The minor $M'/2$ is $M$.  The minor $M'\setminus 2$ is covered by
the induction hypothesis: if $\psi$ is the map in this matroid,
then for each $\mf m$ the sequence $d_\bullet(\psi_{\mf m})$
is obtained from $d_\bullet(\phi_{\mf m})$ by replacing the final
infinite run of 0s by 1s, so one of the substrings 10 is lost.
The minors $M'/1$ and $M'\setminus 1$ both are handled by case~(e).
This proves case~(f) and finishes our discussion of one-element matroids.
\end{proof}

We approach the reduction of matroids on several elements 
to our basis in Lemma~\ref{lem:G molecule} below, 
in several steps as we did in Lemma~\ref{lem:G atom}.
The bulk of our discussion here will pertain to matroids with one generic basis;
in the terminology of~\cite{D'Adderio-Moci}, these are called \emph{molecules}
(since for matroids over a field all molecules are direct sums of \emph{atoms},
i.e.\ one-element matroids).
First we state two subsidiary technical lemmas.

\begin{lemma}\label{lem:projectives in molecules}
If $M$ is a molecule on ground set~$E$ and $a$ a generic coloop in it, then
$\projpart{M(A)} = \projpart{M(E\setminus a)}\oplus\projpart{M(Aa)}$
for every $A\subseteq E\setminus\{a\}$.
\end{lemma}

\begin{proof}
This is clear for $A=E\setminus\{a\}$, so by induction on the size of the complement of~$A$
we need only establish the statement $A$ given the statement for~$Ab$.  
The rank drop between $M(A)$ and~$M(Ab)$ equals that between $M(Aa)$ and~$M(Aab)$.
If this rank drop is zero, then we are done because $\projpart{M(Ab)}=\projpart{M(A)}$
and $\projpart{M(Aab)}=\projpart{M(Aa)}$.  If the rank drop is one,
then given maps making a pushout square
\[\xymatrix{
M(A) \ar[r]^\phi\ar[d] & M(Ab) \ar[d] \\
M(Aa) \ar[r]_{\phi'} & M(Aab)
},\]
the kernels of $\phi$ and~$\phi'$ are isomorphic projective modules;
call one of them $P$.
Then, in~$K_0(R)$, we have
\[[M(A)] = [M(Ab)]\oplus [\ker\phi] = [P]\oplus[M(Aab)]\oplus [\ker\phi'] = [P]\oplus[M(Aa)]\]
and the $K$-class of a projective module determines it.
\end{proof}

\begin{lemma}\label{lem:G molecule}
If $M$ is a matroid on ground set $E$ satisfying any of the following, 
then the class $[M]\in K(\Mat{R})$ lies in the span
of the classes of the matroids in~$\mc I$.
\begin{enumerate}\renewcommand{\labelenumi}{(\alph{enumi})}
\item $M$ is a direct sum of one-element matroids.
\item The only generic basis of~$M$ is $\emptyset$.
\item $M$ is a molecule.
\item $M$ is arbitrary.
\end{enumerate}
\end{lemma}

\begin{proof}
As in the last proof, we will manipulate our matroid $M$ by fabricating
larger matroids in which $M$ appears as a minor.  But we have more
room to maneuver, as matroids on fewer elements than~$M$ may
also appear in the deletion-contraction relations,
and induction on the ground set size shows that the classes of these are
in the span of~$\mc I$.

\noindent\emph{To (a)}.
For this step of the argument we will use another induction, on
the number of one-element summands of a matroid~$M$ which 
are not $L_0$ or~$L^*_0$, possibly excluding \emph{one} summand
of each of the forms $L_{\mc E}$ and~$L^*_{\mc F}$.  (There may also be a summand
which is an empty matroid for a torsion module; this will be inert
and have no effect on our argument).  
In the base case, $M$ is a direct sum of some copies of~$L_0$, possibly
a single $L_{\mc E}$, some copies of~$L^*_0$, possibly a single $L^*_{\mc E}$,
and some empty matroid $\emptyset_N$; this is an element of~$\mc I$.  

As inductive step, we will use deletion-contraction relations
to increase the number of such summands in two ways,
one of which applies to any direct sum of at least two
one-element matroids.  One of our constructions will replace
a direct sum of two one-element summands of the same generic rank
by a matroid of form $L_{\mc E}\oplus L_0$ or $L^*_{\mc F}\oplus L^*_0$.
The other will replace a direct sum of two one-element summands
of unequal generic ranks with some $L_{\mc E}\oplus L^*_{\mc F}$.

For the former construction
suppose we have a two-element molecule $N$,
without loss of generality having two coloops,
which is a summand of~$M$; write $M=N\oplus K$.
For convenience suppose the ground set of~$N$ is $\{1,2\}$.

The basis of $N$ is~$\emptyset$, and $N(\emptyset)$
has the form $P\oplus R\oplus T$, where $P$ is a rank~1 projective module
and $T$ is some torsion module.
Fix maps $\phi:N(\emptyset)\to N(1)$, $\psi:N(\emptyset)\to N(2)$.
By making the non-free analogue of a change of basis 
in this splitting $P\oplus R$ if necessary,
we can suppose that neither of the saturations of $\ker\phi$ nor $\ker\psi$ is contained
in $R$ or~$P$.  Now embed $N$ in a realizable matroid $N'$ on $\{1,2,3,4\}$
so that $N=N'\setminus\{3,4\}$, the map $N'(\emptyset)\to N(3)$ is
the quotient map $P\oplus R\oplus T\to L_{[P]}(1)\oplus R\oplus T$ on the first factor,
the map $N'(\emptyset)\to N(4)$ is the quotient $P\oplus R\oplus T\to P\oplus T$
on the second, and the rest of $N'$ is completed by taking pushouts.  

By construction, none of the kernels of the maps with source $N'(\emptyset)$
in this realization has its saturation contained in another such saturation,
so that the quotient of $N'(\emptyset)$ by the sum of two such kernels has
rank~0.  Thus, the generic matroid of~$N'$ is $U_{2,4}$.  
Thus, the direct sum $M'=N'\oplus K$ is $U_{2,4}$ plus a molecule.
We will use deletion-contraction relations 
to break $M'$ down in two ways, the knowledge of the generic matroid of~$M'$
assuring us that we are not choosing loops or coloops.  
On one hand, use (in sequence) the elements $3$ of~$M'$,
$4$ of~$M'\setminus 3$, $1$ of~$M'\setminus 3/4$, 
$1$ of~$M'/3$, and $2$ of $M'\setminus 1/3$.
On the other, use the elements $1$ of~$M'$,
$2$ of~$M'\setminus 1$, $3$ of~$M'\setminus 1/2$, 
$3$ of~$M'/1$, and $4$ of $M'\setminus 3/1$.
This gives us equalities of classes in the Grothendieck ring:
\begin{align*}
\noequals [M'\setminus3,4]+[M'\setminus1,3/4]+[M'\setminus3/1,4]+[M'\setminus 1,2/3]+[M'\setminus 1/2,3]+[M'/1,3]
\\&= [M']
\\&= [M'\setminus1,2]+[M'\setminus1,3/2]+[M'\setminus1/2,3]+[M'\setminus 3,4/1]+[M'\setminus 3/1,4]+[M'/1,3]
\end{align*}
The term $[M'/1,3]$ cancels, and all of the remaining terms aside from 
$[M'\setminus3,4]$ and $[M'\setminus1,2]$ are classes of matroids on fewer elements,
so they are in the span of the classes of~$\mc I$
by our top-level induction.  The matroid $M'\setminus 3,4$ is our original~$M$.
Finally, $M'\setminus 1,2$ has more summands than~$M$ which are $L_0$ or~$L_0^*$:
there is a new such summand in $M'\setminus 1,2$ on the element~$4$.  
So it is covered by one of our inductions as well.

Turning to the latter construction, we will in fact need to invoke a second induction,
on the rank of the generic matroid of~$M$, that is the size of its generic basis.
We set this up decreasingly, so the base case is when $M$ has only coloops:
in this case, $M$ has no loop and this construction cannot in fact imply.  

Continuing, we suppose $M$ has a
two-element summand $N$, say on ground set $\{1,2\}$,
which is itself the sum of a matroid $N_1$ on its coloop~$1$,
and a matroid $N_2$ on its loop~$2$.  
Again we write $M=N\oplus K$.

By choosing any maps and computing the pushout, we may construct
a matroid $\tilde N_2$ on ground set $\{2,4\}$ where $\tilde N_2(\emptyset) = N_2(\emptyset)$,
$\tilde N_2(2) = N_2(2)$, and $\tilde N_2(4) = \tors{N_2(\emptyset)}\oplus L_{[P]}(1)$
where $P = \projpart{(N_2(\emptyset))}$.  Its generic matroid will be $U_{1,2}$.
With the dual of this construction we also construct a matroid
$\tilde N_1$ on ground set $\{1,3\}$ with generic matroid $U_{1,2}$,
where $\tilde N_1(3) = N_1(\emptyset)$, $\tilde N_1(13) = N_1(1)$, 
and $\tilde N_1(1) = N_1(1)\oplus L_{\mc E}(1)$
where $\mc E = [N_1(\emptyset)]-[N_1(1)]$.  

We will construct $N'$ as a perturbation of $\tilde N\doteq\tilde N_1\oplus\tilde N_2$, as follows.
Fix realizations of $\tilde N_1$ and~$\tilde N_2$, so that the induced realization
of $\tilde N$ provides four maps $\phi_1,\ldots,\phi_4$
with cyclic kernel from the module $\tilde N(\emptyset)$,
corresponding respectively to the atoms $1,\ldots,4$ covering $\emptyset$ in $\mc B(4)$.
The kernels of $\phi_1$ and~$\phi_3$ are both contained in~$\tilde N_1(\emptyset)$,
while the kernels of $\phi_2$ and~$\phi_4$ are contained in~$\tilde N_2(\emptyset)$;
all of them are isomorphic to $R$ as $R$-modules.
The module $\tilde N_1(\emptyset)$ is the direct sum of a projective rank 1 summand~$P$,
and a torsion module.  There exists an injection $\psi:P\hookrightarrow \ker\phi_2\cap\ker\phi_4$.
This can be composed with the embedding $\ker\phi_2\cap\ker\phi_4\subseteq \tilde N(\emptyset)$
and summed with zero maps on the other summands to produce a map
$\psi:N(\emptyset)\to N(\emptyset)$.  
The map $(\id+\psi):N(\emptyset)\to N(\emptyset)$ is then ``upper triangular'' and hence an automorphism.  
Let $x$ be a generator of $\ker\phi_3$, and
define a new map $\phi'_3$ from $\tilde N(\emptyset)$
to be the quotient by the submodule $\<x+\psi(x)\>$.  Finally, let $N'$
be the matroid on ground set~$\{1,2,3,4\}$ with 
$N'(\emptyset) = \tilde N(\emptyset)$ and whose maps and other modules
induced as pushouts of $\phi_1$, $\phi_2$, $\phi'_3$, and $\phi_4$.

Our perturbation of $\phi_3$ to $\phi'_3$ has arranged that $\cork_{N'}(13) = 0$.
On the other hand, is $3\not\in A$ then $N'(A)$ is unchanged from $\tilde N(A)$;
if $A$ contains 3 and one of 2 or~4 but not~1 then $N'(A)\cong\tilde N(A)$
by construction of~$\psi$; and $N'(3)\cong\tilde N(3)$ as well, since $\id+\psi$
is an automorphism.  In particular the generic matroid of~$N$ is the rank~2 matroid
on $\{1,2,3,4\}$ with no loops whose only nontrivial parallelism class is $\{2,4\}$.

Let $M'=N'\oplus K$.  
We have deletion-contraction relations giving the following equalities:
\begin{align*}
\noequals [M'\setminus 3,4]+[M'\setminus 3/4]+[M'\setminus 1,2/3]+[M'\setminus 1/2,3]+[M'/1,3]
\\&= [M']
\\&= [M'\setminus 1,2]+[M'\setminus 1/2]+[M'\setminus 3,4/1]+[M'\setminus 3/1,4]+[M'/1,3]
\end{align*}
The term $[M'/1,3]$ cancels, and the two terms before it in each line
are matroids on fewer elements.  The matroid $M'\setminus 3/4$ is our original~$M$,
since $M$ was the same minor of~$\tilde N\oplus K$ and we haven't altered the
relevant modules in it.
The matroid $N'\setminus 1/2$, for the analogous reason,
is the direct sum of $L^*_{[P]}$ on the element~4, 
$L_{\mc E}$ on the element~3, and an empty matroid, so 
$M'\setminus 1/2$ improves on the quantity counted in the induction
we introduced at the start of this case~(a).
The remaining matroids, $M'\setminus 3,4$ and $M'\setminus 1,2$,
are also direct sums of one-element matroids, and they both have
generic rank~2, so they are covered by our latest-introduced induction.
Altogether, this finishes case~(a).

\noindent\emph{To (b)}.
We will use induction on the number of elements of~$E$
which are not the ground set of a one-element direct summand.  
The base case is part~(a).

We construct a matroid $M'$ on $E\amalg\{\eta\}$ which will agree in most of its modules
with the direct sum of~$M$ and a loop $\emptyset\mapsto 0$, $\{\eta\}\mapsto 0$.
In particular $M'/\eta$ will be $M$.  
We let $M'(\emptyset)$ be obtained from $M(\emptyset)$ by replacing its
largest invariant factor with a projective module with the same determinant.
For each $b\in E$, we use Lemma~\ref{lem:disjoint} to produce a cyclic module
$C(b)$ of disjoint support from any module in~$M$ and so that 
$[C(b)]+[M(b)] = [M(\emptyset)]$ in~$\Pic(R)$, and then set 
$M'(b) = M(b)\oplus C(b)$.  In any other case set $M'(A) = M(A\setminus\eta)$,
where $\eta$ is not necessarily in~$A$.

Our choices of $M'(\emptyset)$ and the modules $M'(b)$ for singletons
are exactly as is needed so that all the pairs $M'(\emptyset)$, $M'(b)$
satisfy the $K$-theoretic condition of Proposition~\ref{p:Dedekind M1}.  
For the other covering relations of subsets of~$E\amalg\{\eta\}$,
both modules are rank~0 so the $K$-theoretic condition is trivially satisfied.
The localization conditions are essentially inherited from~$M$.
Since the summands $C(b)$ have support disjoint from any of the other
modules under consideration, they don't interfere in this respect.  
The alteration we have made to~$M'(\emptyset)$ replaces a final infinite
string of $0$s by~$1$s in the sequences $d_\bullet$ associated to the maps
$M'(\emptyset)_{\mf m}\to M'(b)_{\mf m}$; the resulting sequence is still
of the sort allowed by Proposition~\ref{p:DVR M1}.
These same facts about the localizations also suffice to establish
Proposition~\ref{p:Dedekind M2}, in which only the
local considerations of Proposition~\ref{p:DVR M2} are relevant.

The generic matroid of $M'$ is $U_{1,|E|+1}$.
Therefore, no deletion of~$M'$ with more than one element
is a molecule, and we may freely use the deletion-contraction relation
on such deletions.  
Let $a$ be any element of~$E$.  Splitting $M$ into three minors
by deletion-contraction in two ways, we have
\begin{align*}
  [M'\setminus\eta\setminus a] + [M'\setminus\eta/a] + [M'/\eta]
= [M'] 
= [M'\setminus a/\eta] + [M'\setminus a/\eta] + [M'/a]
\end{align*}
so that, cancelling the common deletion,
\[[M'\setminus\eta/a] + [M'/\eta] = [M'\setminus a/\eta] + [M'/a].\]
Here, the minor $M'/\eta$ is our matroid~$M$ of interest.
The matroid $M'/a$ has a one-element summand with ground set $\{\eta\}$
together with whichever one-element summands $a$ had, so it is subsumed by
our induction hypothesis.  The other two matroids are on fewer elements.
This proves case~(b).

\noindent\emph{To (c)}.
Here we use induction on the number of generic coloops
and on the size of the number of elements
which don't generate single-element direct summands.

Suppose that $a$ is a generic coloop of~$M$.  Then $M(E\setminus a)$
has a rank~1 projective summand, call it $P$.  By Lemma~\ref{lem:essential}, 
the empty matroid $\emptyset_P$ for~$P$ splits as a direct summand of~$M\setminus a$.
Name the other direct summand $N$.  

Let $C$ be a cyclic module which is sufficiently large that every cyclic summand
of a module appearing in~$M$ is isomorphic to a quotient of~$C$, 
and such that $[P] = [C]$ in~$\Pic(R)$.  Let $M'$ be a system of $R$-modules
so that $M'\setminus\eta = M$; $M'/\eta\setminus a = N\oplus \emptyset_C$; 
and $M'/\eta/a = M/a$.  That is, 
$M'$ is obtained from the direct sum $\tilde M$
of~$M$ and the one-element matroid $\emptyset\mapsto 0$, $\{\eta\}\mapsto 0$
by replacing a summand~$P$ by~$C$ at every set containing $\eta$ but not~$a$.
We will show that $M'$ is a matroid
using Propositions \ref{p:Dedekind M1} and~\ref{p:Dedekind M2}.

For Proposition~\ref{p:Dedekind M2}, since $\tilde M$ is a matroid, we need only check
that the replacements of $P$ by~$C$
don't interfere with the condition to be checked in Proposition~\ref{p:DVR M2}.
If $\mf m$ is a maximal prime, then the sequences $d_\bullet(\tilde M(A))$ 
and $d_\bullet(M'(A))$ are of course identical if no replacement has
taken place, and if one has, they differ only in that
$d_i(M'(A)) = d_i(\tilde M(A))-1$ for all $i\geq k$, where $k$
is such that every sequence $d_\bullet(M(B))$ is constant from the $k$\/th position on.
Replacing $P$ by~$C$ can't cause any difference $d_i(M'(A))-d_i(M'(Ab))$
to leave the range $\{0,1\}$: 
if this difference were to be~$2$ then $b$ must be~$\eta$, and if it were to be~$-1$
then $b$ must be $a$, but neither of these situations occur in the construction.
The replacement also doesn't change the quantity on the left side of the displays
in (L2a) and~(L2b) for any two-element minor $M''$ of~$M'$, 
and hence doesn't undermine the truth of these conditions,
unless the ground set of~$M''$ is $\{a,\eta\}$,
in which case that quantity is incremented.  But in this event, 
by construction, the equality of~(L2b) is attained in the corresponding minor
of $\tilde M$ for $d_{\leq k}$, and so (L2b) is still true of~$M''$.

For Proposition~\ref{p:Dedekind M1}, all that remains to check are
the equalities of determinants.
There are two cases to consider which are not inherited from $M$ or $N\oplus \emptyset_C$.  
One involves $M'(A)$ and $M'(A\eta)$ for $\eta\not\in A$ and $a\not\in A$,
where the rank drop is~1,
and $M'(A) = P\oplus N(A)$ and $M'(A\eta) = C\oplus N(A)$ have the same determinant
by choice of~$C$.  The other involves $M'(A)$ and $M'(Aa)$ for $\eta\in A$ and $a\not\in A$,
where the rank drop is~0.  In this case Lemma~\ref{lem:projectives in molecules}
gives that $\projpart{M(A)} = P\oplus\projpart{M(Aa)}$.
Then $\projpart{M'(A)} = \projpart{(C\oplus M(Aa))} = \projpart{M(Aa)}$
and $\projpart{M'(Aa)} = \projpart{M(Aa)}$ agree.

Thus $M'$ is a matroid.  In its generic matroid, all elements
are loops or coloops except for $\eta$ and~$a$ which generate a uniform matroid $U_{1,2}$.
so we have deletion-contraction relations
\[[M'\setminus\eta] + [M'/\eta] = [M'] = [M'\setminus a] + [M'/a].\]
In this relation $M'\setminus\eta$ is our~$M$.
The matroid $M'/\eta$ has a one-element direct summand on ground set~$\{a\}$,
so is encompassed by our second induction; the matroids $M'\setminus a$
and $M'/a$ have a greater number of coloops than~$M$, so are encompassed by our first.
This proves case~(c).

\noindent\emph{To (d)}.
Repeatedly using deletion-contraction to break up
any matroid with at least two bases, on any element which is not a loop or coloop, 
expresses the class of any matroid as a sum of classes of molecules.
\end{proof}

\subsection{Arithmetic Tutte polynomial and quasi-polynomial}\label{ssec:arithmetic Tutte}
In this subsection, $M$ is a matroid over~$\Z$.  
We show that the arithmetic Tutte polynomial and the Tutte quasi-polynomial
are images of $\T_M$ under ring homomorphisms.

Since $\Z[\fgMod{\Z}]\otimes\Z[\fgMod{\Z}]\simeq \Z[X,Y]$, we have that
$$\T_M = \sum_{A\subseteq E}
(X^R)^{\cork_M(A)} (Y^R)^{\operatorname{nullity}_M(A)} (XY)^{\tors{M(A)}}.
$$
where we use the notation $\operatorname{nullity}_M(A) =\cork_{M^*}(E\setminus A)=\dim M^*(E\setminus A)$.

We may define a specialization of $\T_M$ by evaluating $X^R$ at $(x-1)$,
$Y^R$ at $(y-1)$, and $(XY)^N$ at the cardinality of~$N$ for each torsion module~$N$.
This specialization is the arithmetic Tutte polynomial $\M_{\hat{M}}(x,y)$ of the quasi-arithmetic matroid $\hat{M}$ defined by $M$: 
$$
\M_{\hat{M}}(x,y)= \sum_{A \subseteq E} m(A)(x-1)^{\rk(E)-\rk(A)}(y-1)^{|A|-\rk(A)},
$$
where $m(A) = |\tors{M(A)}|$.
This polynomial proved to have several applications to toric arrangements, partition functions, zonotopes, and graphs with labeled edges (see \cite{MociT}, \cite{D'Adderio-Moci}). Notice that an ordinary matroid $\tilde M$ can be trivially made into an arithmetic matroid $\hat{M}$ by setting all the multiplicities to be equal to 1, and then 
$\M_{\hat{M}}(x,y)$ is nothing but the classical Tutte polynomial $\T_{\tilde{M}}(x,y)$.

Clearly, the polynomial $\M_{\hat{M}}(x,y)$ is not the universal deletion-contraction 
invariant of $\hat{M}$.  For instance, the 
ordinary Tutte polynomial $\T_{\tilde{M}}(x,y)$ of the matroid $\tilde M$ obtained from $\hat{M}$ by 
forgetting of its arithmetic data is also a deletion-contraction invariant 
of $\hat{M}$, which is not determined by $\M_{\hat{M}}(x,y)$. This fact led the authors of 
\cite{Branden-Moci} to define a \emph{Tutte quasi-polynomial} $\mathbf{Q}_M(x,y)$, interpolating 
between $\T_{\tilde{M}}(x,y)$ and $\M_{\hat{M}}(x,y)$. This invariant is stronger, but still 
not universal, and more importantly, it is not an invariant of the arithmetic matroid,
as it depends on the groups $\tors{M(A)}$ and not just on their cardinalities. We will now show that $\mathbf{Q}_M(x,y)$ is actually an invariant of the matroid over $\Z$, and write explicitly how to compute it from the universal invariant.

For every positive integer $q$, let us define a function $V_q$ as follows:
\[
V_q((XY)^{\Z/p^k}) = \begin{cases}
1 & \mbox{if $p^k$ divides $q$} \\
p^{k-j} & \mbox{if $0\leq j<k$ is maximal s.t.\ $p^j$ divides $q$.}
\end{cases}
\]
We will extend this to define $V_q((XY)^N)$ multiplicatively for any torsion abelian group $N$.
Then we define a specialization of $\T_M$ to the ring of quasipolynomials
by specializing $X^R$ to $(x-1)$, $Y^R$ to $(y-1)$, and $(XY)^N$ to $V_{(x-1)(y-1)}((XY)^N)$.
This gives
$$
\mathbf{Q}_M(x,y)\doteq
\sum_{A\subseteq E}(x-1)^{\cork_{M(A)}} (y-1)^{\operatorname{nullity}_{M(A)}} V_{(x-1)(y-1)}((XY)^{\tors{M(A)}})=$$
$$=\sum_{A \subseteq E} \frac{|\tors{M(A)}|}{|(x-1)(y-1) \tors{M(A)}|}(x-1)^{\rk(E)-\rk(A)}(y-1)^{|A|-\rk(A)}.
$$

Since $(q+|G|)G=qG$ holds for any finite group $G$, the function $
\mathbf{Q}_M(x,y)$ is a quasi-polynomial in $q=(x-1)(y-1)$. In particular, when $|\tors{M(A)}|$ divides $(x-1)(y-1)$, then the group $(x-1)(y-1) \tors{M(A)}$ is trivial and $\mathbf{Q}_M(x,y)$ coincides with $\M_{\hat{M}}(x,y)$; while when $|\tors{M(A)}|$ is coprime with $(x-1)(y-1)$, then $\mathbf{Q}_M(x,y)$ coincides with $\T_{\tilde{M}}(x,y)$. Then in some sense $\mathbf{Q}_M(x,y)$ interpolates between the two polynomials.

Notice that while $\M_{\hat{M}}$ and $\T_{{\tilde M}}(x,y)$ only depend on the induced quasi-arithmetic matroid $\hat{M}$, $\T_M$ and $\mathbf{Q}_M(x,y)$ are indeed invariants of the matroid over $\Z$, $M$. Also the \emph{chromatic quasi-polynomial} and the \emph{flow quasi-polynomial} defined in \cite{Branden-Moci} are actually invariants of the matroid over $\Z$: by \cite[Theorem 9.1]{Branden-Moci} they are specializations of $\mathbf{Q}_M(x,y)$, and hence of the universal invariant 
$\T_M$.

\end{document}